\pdfoutput=1
\documentclass[10pt,oneside]{amsart}
\usepackage[
paper=a4paper,
headsep=20pt,text={136mm,208mm},centering,includehead
]{geometry}
\usepackage{graphicx}
\usepackage[dvipsnames,usenames]{xcolor}
\usepackage[colorlinks=true,urlcolor=ForestGreen,linkcolor=ForestGreen,citecolor=ForestGreen]{hyperref}
\hypersetup{
	linkbordercolor={1 0 0}, 
	citebordercolor={0 1 0} 
}
\usepackage{wasysym}
\usepackage{amssymb,mathtools,mathrsfs}
\usepackage{bbm}
\usepackage[UKenglish]{babel}
\usepackage[noabbrev]{cleveref}
\usepackage{tikz}
\usetikzlibrary{trees,decorations.pathmorphing,decorations.markings,matrix,shapes}
\usepackage[labelformat=simple]{subcaption}

\usepackage[backgroundcolor=green,linecolor=green,bordercolor=green]{todonotes}
\usepackage{enumerate}
\usepackage{csvsimple}
\usepackage{longtable}

\iftrue
\makeatletter
\def\@settitle{%
	\vspace*{-20pt}
	\begin{flushleft}%
		\baselineskip14\p@\relax
		\normalfont\bfseries\LARGE
		\@title
	\end{flushleft}%
}
\def\@setauthors{%
	\begingroup
	\def\thanks{\protect\thanks@warning}%
	\trivlist
	\large \@topsep30\p@\relax
	\advance\@topsep by -\baselineskip
	\item\relax
	\author@andify\authors
	\def\\{\protect\linebreak}%
	\authors
	\ifx\@empty\contribs
	\else
	,\penalty-3 \space \@setcontribs
	\@closetoccontribs
	\fi
	\normalfont
	\@setaddresses
	\endtrivlist
	\endgroup
}
\def\@setaddresses{\par
	\nobreak \begingroup\raggedright
	\small
	\def\author##1{\nobreak\addvspace\smallskipamount}%
	\def\\{\unskip, \ignorespaces}%
	\interlinepenalty\@M
	\def\address##1##2{\begingroup
		\par\addvspace\bigskipamount\noindent
		\@ifnotempty{##1}{(\ignorespaces##1\unskip) }%
		{\ignorespaces##2}\par\endgroup}%
	\def\curraddr##1##2{\begingroup
		\@ifnotempty{##2}{\nobreak\noindent\curraddrname
			\@ifnotempty{##1}{, \ignorespaces##1\unskip}\/:\space
			##2\par}\endgroup}%
	\def\email##1##2{\begingroup
		\@ifnotempty{##2}{\smallskip\nobreak\noindent E-mail address%
			\@ifnotempty{##1}{, \ignorespaces##1\unskip}\/:\space
			\ttfamily##2\par}\endgroup}%
	\def\urladdr##1##2{\begingroup
		\def~{\char`\~}%
		\@ifnotempty{##2}{\nobreak\noindent\urladdrname
			\@ifnotempty{##1}{, \ignorespaces##1\unskip}\/:\space
			\ttfamily##2\par}\endgroup}%
	\addresses
	\endgroup
	\global\let\addresses=\@empty
}
\def\@setabstracta{%
	\ifvoid\abstractbox
	\else
	\skip@25\p@ \advance\skip@-\lastskip
	\advance\skip@-\baselineskip \vskip\skip@
	\box\abstractbox
	\prevdepth\z@ 
	\vskip-15pt
	\fi
}
\renewenvironment{abstract}{%
	\ifx\maketitle\relax
	\ClassWarning{\@classname}{Abstract should precede
		\protect\maketitle\space in AMS document classes; reported}%
	\fi
	\global\setbox\abstractbox=\vtop \bgroup
	\normalfont\small
	\list{}{\labelwidth\z@
		\leftmargin0pc \rightmargin\leftmargin
		\listparindent\normalparindent \itemindent\z@
		\parsep\z@ \@plus\p@
		
	}%
	\item[\hskip\labelsep\bfseries\abstractname.]%
}{%
	\endlist\egroup
	\ifx\@setabstract\relax \@setabstracta \fi
}

\def\ps@headings{\ps@empty
	\def\@evenhead{%
		\setTrue{runhead}%
		\normalfont\scriptsize
		\rlap{\thepage}\hfill
		\def\thanks{\protect\thanks@warning}%
		\leftmark{}{}}%
	\def\@oddhead{%
		\setTrue{runhead}%
		\normalfont\scriptsize
		\def\thanks{\protect\thanks@warning}%
		\rightmark{}{}\hfill \llap{\thepage}}%
	\let\@mkboth\markboth
}\ps@headings

\def\section{\@startsection{section}{1}%
	\z@{-1.2\linespacing\@plus-.5\linespacing}{.8\linespacing}%
	{\normalfont\bfseries\Large}}
\def\subsection{\@startsection{subsection}{2}%
	\z@{-.8\linespacing\@plus-.3\linespacing}{.3\linespacing\@plus.2\linespacing}%
	{\normalfont\bfseries\large}}
\def\subsubsection{\@startsection{subsubsection}{3}%
	\z@{.7\linespacing\@plus.1\linespacing}{-1.5ex}%
	{\normalfont\bfseries}}
\def\@secnumfont{\bfseries}
\makeatother
\fi 

\makeatother

\newtheorem{theorem}{Theorem}[section]
\newtheorem*{theorem*}{Theorem}

\newtheorem{proposition}[theorem]{Proposition}
\newtheorem{corollary}[theorem]{Corollary}
\newtheorem{lemma}[theorem]{Lemma}

\theoremstyle{definition}
\newtheorem{definition}[theorem]{Definition}

\newtheorem*{remark}{Remark}

\numberwithin{equation}{section}

\newcommand{\kh}{{Kh}}
\newcommand{\cdkh}{CDKh}
\newcommand{\dkh}{DKh}
\newcommand{\vckh}{vCKh}
\newcommand{\vkh}{vKh}

\newcommand{\sg}{\mathfrak{s}}
\newcommand{\vup}{v^{\text{u}}_+}
\newcommand{\vum}{v^{\text{u}}_-}
\newcommand{\vlp}{v^{\text{l}}_+}
\newcommand{\vlm}{v^{\text{l}}_-}

\DeclareRobustCommand{\CloseDef}{%
	\leavevmode\unskip\penalty9999 \hbox{}\nobreak\hfill
	\quad\hbox{$\lozenge$}%
}

\begin{document}
	
\vspace*{-50pt}
\title{Computations of the slice genus of virtual knots}
\author{William Rushworth}
\address{Department of Mathematical Sciences, Durham University, United Kingdom}
\email{\href{mailto:william.rushworth@durham.ac.uk}{william.rushworth@durham.ac.uk}}

\def\subjclassname{\textup{2010} Mathematics Subject Classification}
\expandafter\let\csname subjclassname@1991\endcsname=\subjclassname
\expandafter\let\csname subjclassname@2000\endcsname=\subjclassname
\subjclass{57M25, 57M27, 57N70}

\keywords{Rasmussen invariant, virtual knot concordance, slice genus}

\dedicatory{This paper subsumes the (now withdrawn) arXiv submission \href{https://arxiv.org/abs/1603.02893}{On the virtual Rasmussen invariant}.}

\begin{abstract}
A virtual knot is an equivalence class of embeddings of \(S^1\) into thickened (closed oriented) surfaces, up to self-diffeomorphism of the surface and certain handle stabilisations. The slice genus of a virtual knot is defined diagrammatically, in direct analogy to that of a classical knot. However, it may be defined, equivalently, as follows: a representative of a virtual knot is an embedding of \(S^1\) into a thickened surface \( \Sigma_g \times I \); what is the minimal genus of oriented surfaces \( S \hookrightarrow M \times I \) with the embedded \(S^1\) as boundary, where \( M \) is an oriented \(3\)-manifold with \( \partial M = \Sigma_g \)?

We compute and estimate the slice genus of all virtual knots of \(4\) classical crossings or less. We also compute or estimate the slice genus of \(46\) virtual knots of \(5\) and \(6\) classical crossings whose slice status is not determined in the work of Boden, Chrisman, and Gaudreau. The computations are made using two distinct virtual extensions of the Rasmussen invariant, one due to Dye, Kaestner, and Kauffman, the other due to the author. Specifically, the computations are made using bounds on the two extensions of the Ramussen invariant which we construct and investigate. The bounds are themselves generalisations of those on the classical Rasmussen invariant due, independently, to Kawamura and Lobb. The bounds allow for the computation of the extensions of the Rasmussen invariant in particular cases. As asides we identify a class of virtual knots for which the two extensions of the Rasmussen invariant agree, and show that the extension due to Dye, Kaestner, and Kauffman is additive with respect to the connect sum.
\end{abstract}

\maketitle

\section{Introduction}
\subsection{Statement of results}\label{Subsec:results}
A \emph{virtual knot} is an equivalence class of embeddings of \(S^1\) into thickened (closed oriented) surfaces, up to self-diffeomorphism of the surface and handle stabilisations whose attaching spheres do not intersect the embedded \( S^1 \); \emph{virtual links} are defined analogously \cite{Kauffman1998}. They are represented diagrammatically using knot diagrams with an extra crossing decoration, the \emph{virtual crossing} \raisebox{-2.5pt}{\includegraphics[scale=0.3]{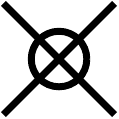}}, up to the virtual Reidemeister moves; see \Cref{Fig:vknot} for such a diagram.

The slice genus of a virtual knot is defined in direct analogy to that of classical knots (see \Cref{Sec:virtualcobordism}); it is less well-studied than that of classical knots, but obstructions to sliceness of virtual knots have been developed by a number of authors. They include the index polynomial of Heinrich \cite{Henrich2008} and the graded genus of Turaev \cite{Turaev2006}. Boden, Chrisman, and Gaudreau \cite{Boden2017} have used these invariants and others to compute or estimate the slice genus of a very large number of the \( 92800 \) virtual knots of \(6\) crossing or less (as given in Green's table \cite{Green}).

In another direction, Manturov and Fedoseev have produced slice obstructions for free knots \cite{Manturov2010,Fedoseev2017,Fedoseev2017a}. A free knot is an equivalence class of \(4\)-valent graphs, and a Gauss code representing a virtual knot may be projected to a code representing a free knot by forgetting the signs and directions of its chords. Given a free knot \( \Gamma \), obstructing the sliceness of \( \Gamma \) necessarily obstructs the sliceness of every virtual knot which projects to it.

We shall focus on the Rasmussen invariant. It has been extended to virtual knots in two different ways, producing two distinct Rasmussen-like invariants: the \emph{virtual Rasmussen invariant} due to Dye, Kaestner, and Kauffman \cite{Dye2014}, and the \emph{doubled Rasmussen invariant} due to the author \cite{Rushworth2017}. Both of these extensions provide obstructions to the sliceness of virtual knots (again see \Cref{Sec:virtualcobordism}).

In this paper we employ these extensions of the Rasmussen invariant to compute or estimate the slice genus of virtual knots. The extensions themselves are derived from two distinct generalisations of Khovanov homology to virtual links, reviewed in \Cref{Sec:review}. The results of the computations are given in two tables which begin on \cpageref{Tab:slicegenera} and \cpageref{Tab:slicegaps} respectively, and are outlined in \Cref{Subsub:results}.

Let \( K \) be a virtual knot; postponing their definition until \Cref{Sec:review}, let \( s( K ) \in 2 \mathbb{Z} \) and \( \mathbbm{s} ( K ) = ( s_1 ( K ), s_2 ( K ) ) \in \mathbb{Z} \times \mathbb{Z} \) denote, respectively, the virtual Rasmussen invariant, and the doubled Rasmussen invariant. Like the classical Rasmussen invariant the quantities \( s ( K ) \) and \( s_1 ( K ) \) are difficult to compute, in general (in constrast \( s_2 ( K ) \) can be computed by hand, as described below). In \Cref{Sec:bounds} four integer quantities are associated to a diagram \( D \) of \( K \) - \( U_v ( D ) \), \( U_d ( D ) \), \( \Delta_v ( D ) \), and \( \Delta_d ( D ) \) - which allow for the estimation of \( s ( K ) \) and \( s_1 ( K ) \).

\begin{theorem*}[\Cref{Thm:delta,Thm:doubledbounds} of \Cref{Sec:bounds}]
	Let \( D \) be a diagram of a virtual knot \( K \). Then
	\begin{equation*}
	U_v ( D ) - 2 \Delta_v ( D ) \leq s ( K ) \leq U_v ( D )
	\end{equation*}
	and
	\begin{equation*}
	U_d ( D ) - \Delta_d ( D ) \leq s_1 ( K ) \leq U_d ( D ).
	\end{equation*}
\end{theorem*}

The bounds \( U_v ( D ) \) and \( U_d ( D ) \) are generalisations of the slice-Bennequin bounds due, independently, to Kawamura \cite{Kawamura2015a} and \cite{Lobb2011} (see \Cref{Subsec:sbbounds}). They are easy to compute for any diagram of \( K \). Further, there are classes of diagrams for which the quantities \( \Delta_v ( D ) \) and \( \Delta_d ( D ) \) simplify. In fact, in \Cref{Sec:deltazero}, we characterise exactly the class of diagrams \(D\) for which \( \Delta_v ( D ) = 0 \) so that \( U_v ( D ) = s ( K ) \).

As an aside, we show that although the extensions \( s \) and \( s_1 \) are distinct in general there is a class of virtual knots on which they agree.

\begin{definition}
	A classical crossing within a virtual knot diagram \( D \) is \emph{even} if it is resolved into its oriented resolution in the alternately colourable smoothing of \( D \); otherwise it is \emph{odd}. A virtual knot diagram is known as \emph{even} if all of its classical crossings are even. A virtual knot is \emph{even} if it possesses an even diagram.\CloseDef
\end{definition}

\begin{remark}
	This definition of odd and even crossings is shown to be equivalent to the standard definition involving Gauss codes in \cite[Proposition \( 4.11 \)]{Rushworth2017}.
\end{remark}

Classically, the oriented smoothing is necessarily alternately colourable (so that every classical knot is even). Virtually, this is no longer the case; consider the diagram given in \Cref{Fig:vknot} (both of its classical crossings are odd).

\begin{theorem*}[\Cref{Cor:evenknots} of \Cref{Sec:review}]
	Let \(K\) be an even virtual knot. Then \( s ( K ) = s_1 ( K ) \).
\end{theorem*}

As a final aside, we show that the virtual Rasmussen invariant is additive with respect to connect sum. By an abuse of notation \( K_1 \# K_2 \) denotes any of the knots which can be obtained as a connect sum between \( K_1 \) and \( K_2 \).

\begin{theorem*}[\Cref{Thm:additive} of \Cref{Sec:cangen}]
	For virtual knots \( K_1 \) and \( K_2 \)
	\begin{equation}
	s ( K_1 \# K_2 ) = s ( K_1 ) + s ( K_2 ).
	\end{equation}
\end{theorem*}

\subsubsection{Results of the computation of \( U_v \) and \( U_d \)}\label{Subsub:results}
\Cref{Sec:applications} contains two tables which give the results of the computation of the bounds \( U_v \) and \( U_d \), along with the results of the computation or estimation of the slice genus which follows (see \Cref{Subsec:estimating}). The first table, beginning on \cpageref{Tab:slicegenera}, contains the results for all virtual knots of \(4\) classical crossings or less, as given in Green's table \cite{Green}. The second table, beginning on \cpageref{Tab:slicegaps}, contains the results for \( 46 \) of the \( 248 \) virtual knots of \( 6 \) classical crossings or less whose slice status is not determined in \cite{BCGtable}.

Many of the calculations and estimations of the virtual and doubled Rasmussen invariants are made by identifying that the knot in question is a connect sum, and applying the additivity of both invariants under that operation.

\subsection{Virtual cobordism}\label{Sec:virtualcobordism}
In direct analogue to those of the classical case we make the following definitions (see \cite{Dye2014} and \cite{Kauffman1998c}). Two virtual knot diagrams \( K_1 \) and \( K_2 \) are \emph{cobordant} if one can be obtained from the other by a finite sequence of births and deaths of circles, oriented saddles, and virtual Reidemeister moves. Such a sequence describes a compact, oriented surface, \( S \), such that \( \partial S = K_1 \sqcup K_2 \). If \( g ( S ) = 0 \) we say that \( K_1 \) and \( K_2 \) are \emph{concordant}. If \( K_2 \) is the unknot, and \( K_1 \) is concordant to \( K_2 \) we say that \( K_1 \) is \textit{slice}. In general, we define the \emph{slice genus} of a virtual knot \( K \), denoted \( g^\ast ( K ) \), as
\begin{equation*}
g^\ast ( K ) = \min \lbrace g ( S ) ~|~ S ~\text{a compact oriented connected surface with}~ \partial S = K \rbrace
\end{equation*}
(here we have simply capped off the unknot in \( \partial S \) with a disc). It is natural to ask whether or not the slice genus of a classical knot may be lowered by treating it as a virtual knot. That is, given a classical knot, does the addition of virtual Redeimeister moves allow one to construct a surface bounding it of lower genus than its classical slice genus? This has been answered in the negative by Boden and Nagel \cite{Boden2016}, a concordance analogue to the result of Goussarov, Polyak, and Viro that classical links are left unaltered if one views them as virtual links \cite{Goussarov1998}.

Behind the scenes, the cobordism surface \( S \) is embedded in a \( 4 \)-manifold of the form \( M \times I \), where \( M \) is a compact, oriented \(3\)-manifold with \( \partial M = \Sigma_k \sqcup \Sigma_l \), where \( \Sigma_i \) denotes a closed oriented surface of genus \( i \). The \( 3 \)-manifold \( M \) is described in the standard way in terms of codimension \( 1 \) submanifolds and critical points: starting from \( \partial M = \Sigma_k \), codimension \( 1 \) submanifolds are \( \Sigma_k \) until we pass a critical point, after which they are \( \Sigma_{k \pm 1} \). Critical points of \( M \) correspond to handle stabilisation. A finite number of handle stabilisations are needed to reach \( \Sigma_l \).

As mentioned in the abstract the slice genus of a virtual knot may be defined in a more natural manner. Let \( K \) be a virtual knot and (by an abuse of notation) let \( K \hookrightarrow \Sigma_g \times I \) be representative of \( K \). Then
\begin{equation*}
	g^\ast ( K ) = \min \left\lbrace g ( S ) ~\left|~	
	\begin{matrix}
		S \hookrightarrow M \times I ~\text{an oriented connected surface with}~ \partial S = K \\
		M ~\text{an oriented \(3\)-manifold with}~ \partial M = \Sigma_g
	\end{matrix}
	\right. \right\rbrace .
\end{equation*}
That this second definition is equivalent to the first follows from the observation that given two representatives of \( K \) in \( \Sigma_g \times I \) and \( \Sigma_{g'} \times I \) with \( g \neq g' \), there exists a cylinder (embedded in a thickened oriented \(3\)-manifold) which cobounds them. Further, this definition highlights the higher-dimensional topology at play when one considers the slice genus of virtual knots. In constrast to the classical case, in which the slice genus of a knot depends only on how surfaces bounding that knot may be embedded into \(B^4\), the slice genus of a virtual knot depends on the surface \(S\) and on the \(3\)-manifold \( M \).

\subsection{The slice-Bennequin bounds}\label{Subsec:sbbounds}
The Rasmussen invariant of a classical knot extracts geometric information from Khovanov homology, yielding a lower bound on the slice genus \cite{Rasmussen2010}. Given a classical knot \( K \) it is, in principle, difficult to compute its Rasmussen invariant, denoted \( s ( K ) \), as it is equivalent to the maximal filtration grading of all elements homologous to a certain generator of the Lee homology of \( K \).

Kawamura \cite{Kawamura2015a} and Lobb \cite{Lobb2011} independently defined diagram-dependent upper bounds on \( s (K) \), denoted \( U ( D ) \) (for \( D \) a diagram of \( K \)), which are easily computable by hand, along with an error term, \( \Delta ( D ) \), the vanishing of which implies that \( s ( K ) = U ( D ) \), in fact. More precisely,
\begin{equation*}
U ( D ) - 2 \Delta ( D ) \leq s ( K ) \leq U ( D ).
\end{equation*}
The bounds \( U ( D ) \) are henceforth referred to as the \emph{strong slice-Bennequin bounds}; in \Cref{Sec:bounds} we construct analogous bounds on the virtual and doubled Rasmussen invariants.

\subsection{Estimating the slice genus}\label{Subsec:estimating}
This paper is concerned with the computation of the slice genus of virtual knots. These computations are achieved using the obstructions to sliceness offered by the two extensions of the Rasmussen invariant mentioned above. As stated, the virtual Rasmussen invariant, and one component of the doubled Rasmussen invariant are difficult to compute (this necessitates the construction of the bounds as mentioned in \Cref{Subsec:results}). The other component of the doubled Rasmussen invariant is, however, readily computable. Precisely, the quantity \( s_2 ( K ) \) can be computed from quickly from any diagram of \( K \), as it is equal to the odd writhe of \( K \). That is:

\begin{theorem*}[[Proposition \( 4.11 \) of \cite{Rushworth2017}]
	Let \( D \) be a diagram of a virtual knot \( K \). Let \( J ( D ) \) denote the sum of the signs of the odd crossings of \( D \). This is a knot invariant, known as the odd writhe of \( K \), and denoted \( J ( K ) \) \cite{Kauffman2004b}. Then \( s_2 ( K ) = J ( K ) \).
\end{theorem*}

\begin{theorem*}[Theorem \( 5.8 \) of \cite{Rushworth2017}]
	Let \( K \) be a virtual knot such that \( s_2 ( K ) \neq 0 \). Then \( K \) is not slice.
\end{theorem*}

Whilst it is more difficult to compute, the other component of the doubled Rasmussen invariant also obstructs sliceness.

\begin{theorem*}[Corollary \(5.5\) of \cite{Rushworth2017}]
	Let \( K \) be a virtual knot with \( s_2 ( K ) = 0 \). If \( s_1 ( K ) \neq 0 \) then \( K \) is not slice.
\end{theorem*}

The virtual Rasmussen invariant provides a lower bound on the slice genus of a virtual knot.

\begin{theorem*}[Theorem \( 5.6 \) of \cite{Dye2014}]
	Let \( K \) be a virtual knot. Then \( | s ( K ) | \leq 2 g^\ast ( K ) \).
\end{theorem*}

The computations and estimations of the slice genus are made as follows. Let \( D \) be the diagram of a virtual knot \( K \) given in Green's table \cite{Green}, then:
\begin{enumerate}[(i)]
	\item Compute \( U_v ( D ) \), \( U_d ( D ) \), \( \Delta_v ( D ) \), \( \Delta_d ( D ) \), and \( s_2 ( K ) \) for \( D \), in order to estimate or compute \( s ( K ) \) and \( s_1 ( K ) \).
	\item Take the greatest of the upper bounds on \( g^\ast ( K ) \) provided by the estimations or computations of \( s ( K ) \), \( s_1 ( K ) \), and \( s_2 ( K ) \).
	\item Attempt to find a cobordism from \( D \) to the unknot of genus equal to the greatest upper bound on \( g^\ast ( K ) \), thus computing \( g^\ast ( K ) \).
	\item Failing that, find a cobordism of higher genus so that a region in which \( g^\ast ( K ) \) lies is identified.
\end{enumerate}

\subsection{Plan of the paper}
First, in \Cref{Sec:review}, we outline the issues faced when extending Khovanov homology to virtual links, and review two distinct ways of overcoming them i.e.\ two extensions of Khovanov homology to virtual links. Further, we review the extensions of the Rasmussen invariant produced from each of the homology theories. We also identify in \Cref{Subsec:evenknots} a class of virtual knots for which the two extensions of the Rasmussen invariant are equal.

Next, in \Cref{Sec:cangen}, we produce canonical chain-level generators of one of the relevant homology theories. This is done by simplifying the decorated diagrammatic generators defined in \cite{Dye2014}, so that elements of the algebraic chain complex may be read off from them.

These canonical generators are required in \Cref{Sec:bounds}, in which we construct the strong slice-Bennequin bounds on both the virtual and the doubled Rasmussen invariant. In this we follow much the same path as Lobb \cite{Lobb2011}; in fact, in the case of the virtual Rasmussen invariant, we recover formulae identical to his. In the case of the doubled Rasmussen invariant, however, the formulae arrived at are substantially different, a consequence of the structural differences between doubled Khovanov homology and its classical predecessor.

Finally, in \Cref{Sec:applications}, we use the tools we have developed to compute or estimate the slice genus of a large portion of the knots given in Green's table \cite{Green}.

\noindent \textbf{Acknowledgements.} We thank A Referee for very helpful comments on an earlier version of this paper, and Hans Boden, Micah Chrisman, and Robin Gaudreau for sharing and discussing their work.

\section{Review}\label{Sec:review}
We review the two homology theories used throughout this work. In an attempt to avoid confusion we shall refer to the theory due to Manuturov and reforumulated by Dye, Kaestner, and Kauffman as MDKK homology, and denote it by \( \vkh \). We denote the other theory in question, doubled Khovanov homology, by \( \dkh \). Classical Khovanov homology, where required, is denoted by \( \kh \). The perturbed versions of the theories are denoted by \( \vkh ' \), \( \dkh ' \), and \( \kh ' \).

The review of MDKK homology contained in \Cref{Sec:DKKreview} is substantially more detailed than the review of doubled Khovanov homology (contained in \Cref{Sec:doubledreview}). This is because the methods used in \Cref{Sec:bounds} require chain-level generators of the complexes \( \vkh ' \) and \( \dkh ' \). We are already in possession of such generators in case of \( \dkh ' \) but not \( \vkh ' \). (In \Cref{Sec:cangen} we construct these generators.)

Before outlining the homology theories we describe the complications one encounters when attempting to extend Khovanov homology to virtual links.

\subsection{Extending Khovanov homology}\label{Subsec:extending}
Manturov first defined Khovanov homology for virtual links \cite{Manturov2006}. His theory was reformulated by Dye, Kaestner, and Kauffman in order to define a virtual Rasmussen invariant \cite{Dye2014}. An alternative extension of Khovanov homology to virtual links is doubled Khovanov homology, which provides the doubled Rasmussen invariant \cite{Rushworth2017}. Here we briefly outline the problems encountered in attempting to extend Khovanov homology to virtual links, and the paths taken in \cite{Dye2014} and \cite{Rushworth2017} to overcome them.

The fundamental obstruction to transferring Khovanov homology to the virtual setting is the existence of the \textit{single-cycle smoothing} depicted in \Cref{Fig:121} (otherwise known as a \textit{one-to-one bifurcation}). If the module assigned to a circle within a smoothing is the same as that assigned by classical Khovanov homology the map associated to this smoothing, denoted \( \eta \), must be identically zero, in order to preserve the quantum grading. This, in turn, causes the face depicted in \Cref{Fig:problemface} to fail to commute. Notice that the differential along the top and right-hand edges is \( \eta \circ \eta = 0 \), but along the left-hand and bottom edges it is \( m \circ \Delta \neq 0 \) so that \( d^2 \neq 0 \).

\begin{figure}
	\centering
	\begin{subfigure}[b]{0.4\textwidth}
		\begin{centering}
			\begin{tikzpicture}[scale=0.8,
			roundnode/.style={}]
			
			\node[roundnode] (s0)at (-2,0)  {
				\includegraphics[scale=0.5]{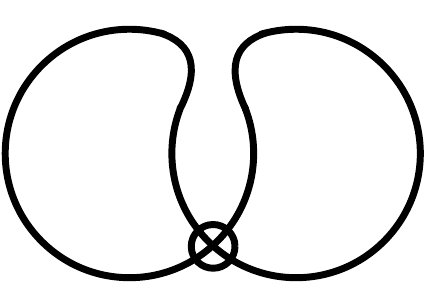}};
			
			\node[roundnode] (s1)at (2,0)  {
				\includegraphics[scale=0.5]{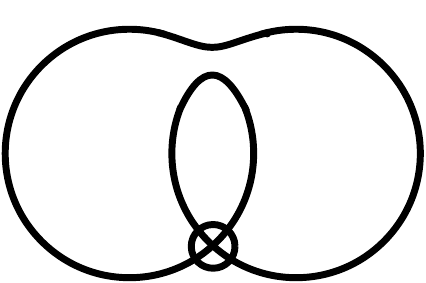}};
			
			\draw[->,thick] (s0)--(s1) node[above,pos=0.5,thick]{\( \eta \)} ;			
			\end{tikzpicture}	
		\end{centering}
		\caption{The single-cycle smoothing.}\label{Fig:121}
	\end{subfigure}
	~
	\begin{subfigure}[b]{0.4\textwidth}
		\begin{centering}
			\begin{tikzpicture}[
			roundnode/.style={}]
			
			\node[roundnode] (s0)at (-2,0)  {
				\includegraphics[scale=0.35]{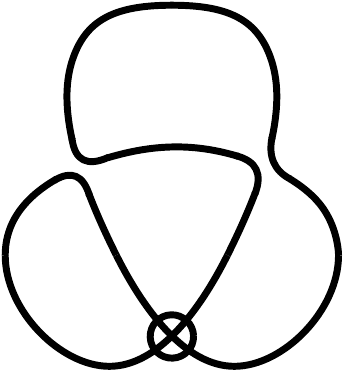}};
			
			\node[roundnode] (s1)at (0,1)  {
				\includegraphics[scale=0.35]{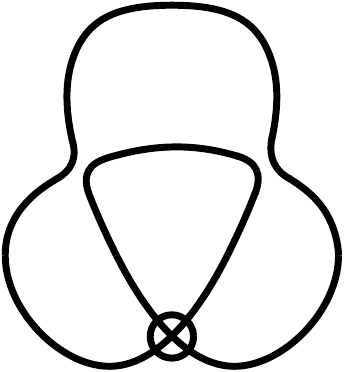}};
			
			\node[roundnode] (s2)at (0,-1)  {
				\includegraphics[scale=0.35]{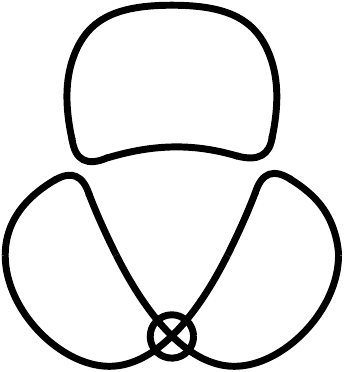}};
			
			\node[roundnode] (s3)at (2,0)  {
				\includegraphics[scale=0.35]{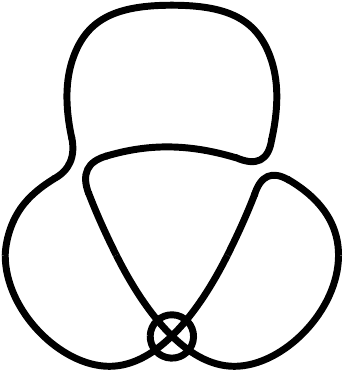}};
			
			\draw[->] (s0)--(s1) node[above,pos=0.5,thick]{\( \eta \)} ;
			
			\draw[->] (s0)--(s2) node[above,pos=0.5,thick]{\( \Delta \)} ;
			
			\draw[->] (s1)--(s3) node[above,pos=0.5,thick]{\( \eta \)} ;
			
			\draw[->] (s2)--(s3) node[above,pos=0.5,thick]{\( m \)} ;
			
			\end{tikzpicture}
		\end{centering}
		\caption{The problem face.}\label{Fig:problemface}
	\end{subfigure}
	\caption{}
\end{figure}

Thus classical Khovanov homology must be augmented in order to detect this face, if one wishes to assign \( \eta \) the zero map. This is the approach taken by Manutrov and subsequently Dye et al, and outlined in \Cref{Sec:DKKreview}. In \cite{Rushworth2017} another approach is taken: the module assigned to a circle within a smoothing is altered, allowing for \( \eta \) to be assigned a non-zero map while being grading preserving. The resulting theory is outlined in \Cref{Sec:doubledreview}.

\begin{remark}
	Tubbenhauer \cite{Tubbenhauer2014a} has constructed a virtual Khovanov homology theory in the manner of Bar-Natan \cite{Bar-natan2005} using non-orientable cobordisms, but there are compatibility issues with the theory presented in \cite{Dye2014}.
\end{remark}

\subsection{Review of MDKK homology}\label{Sec:DKKreview}
We review the construction of MDKK homology and the virtual Rasmussen invariant.

\subsubsection{The complex}

\begin{figure}
	\includegraphics[scale=1]{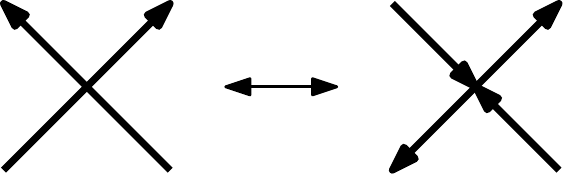}
	\caption{The source-sink decoration.}
	\label{Fig:sourcesink}
\end{figure}

Let \( \mathcal{A} = \mathcal{R}[X]/(X^2-t) \) for \( \mathcal{R} \) a commutative ring and \( t \in \mathcal{R} \). In order to detect the problem face a symmetry present in \( \mathcal{A} \) (which corresponds to the two possible orientations of \( S^1 \)) is exploited using the following automorphism:

\begin{definition}\label{barring} The \emph{barring operator} is the map
\begin{equation}
\overline{\phantom{X}} : \mathcal{A} \rightarrow \mathcal{A},~ X \mapsto -X.
\end{equation}
Applying the barring operator is referred to as \emph{conjugation}.\CloseDef
\end{definition}
Note that if \( \mathcal{R} = \mathbb{R} \) and \( t = -1 \) then \( \mathcal{A} = \mathbb{C} \) and the barring operator is just standard complex conjugation. How the barring operator is applied within the Khovanov complex is determined using an extra decoration on link diagrams, the source-sink decoration as depicted in \Cref{Fig:sourcesink}. A new diagram is formed by replacing the classical crossings with the source-sink decoration, which induces an orientation on the incident arcs of a crossing. Arcs of the diagram on which the induced orientations due to separate crossings disagree are marked by a \emph{cut locus}. We refer the reader to \cite{Dye2014}.

\subsubsection{The virtual Rasmussen invariant}
\label{Sec:lee}

There is a degeneration of Khovanov homology due to Lee \cite{Lee2005}. There is such a degeneration of MDKK homology also. Dye, Kaestner, and Kauffman use the methods of Bar-Natan and Morrison \cite{Bar-Natan2006} to show this. Specifically, they employ the Karoubi envelope of a category and the interpretation of virtual links as abstract links \cite{Carter2000, Kamada2000}, and define the virtual Rasmussen invariant.

As such diagrams are used extensively below, we describe the process given in \cite{Kamada2000} to obtain a (representative of an) abstract link from a (representative of a) virtual link (examples are given in  \Cref{Sec:cangen}). Let \( D \) be a diagram of a virtual link, as in \Cref{Fig:vknot}, then
\begin{enumerate}[(i)]
	\item About the classical crossings place a disc as shown in  \Cref{Fig:ASDCrossing}.
	\item About the virtual crossings place two discs as shown in \Cref{Fig:ASDVirtual}.
	\item Join up these discs with collars about the arcs of the diagram.
\end{enumerate}
The result is a knot diagram on a surface which deformation retracts onto the underlying curve of the diagram. We will denote abstract link diagrams by \( \left( F, D\right) \) for \( D \) a knot diagram and \( F \) a compact, oriented surface (which deformation retracts on to the underlying curve of \( D \)). We treat such diagrams up to stable equivalence, defined below.

\begin{figure}
	\includegraphics[scale=0.75]{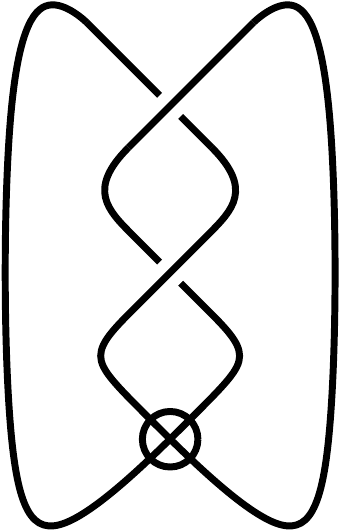}
	\caption{A two-crossing virtual knot diagram.}
	\label{Fig:vknot}
\end{figure}
\begin{figure}
	\includegraphics[scale=1]{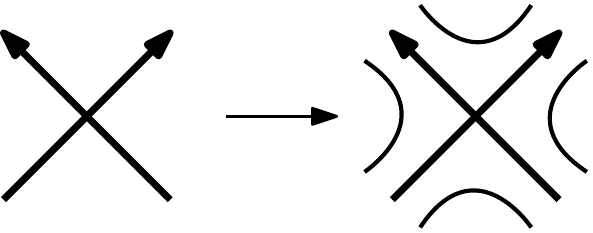}
	\caption{Component of the surface of an abstract link diagram about a classical crossing.}
	\label{Fig:ASDCrossing}
\end{figure}
\begin{figure}
	\includegraphics[scale=1]{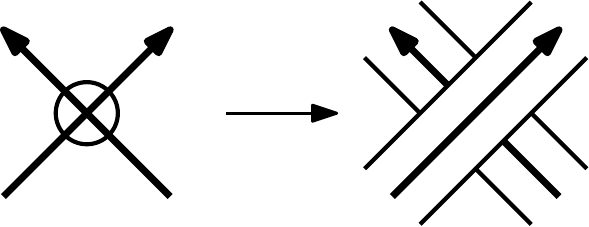}
	\caption{Component of the surface of an abstract link diagram about a virtual crossing.}
	\label{Fig:ASDVirtual}
\end{figure}
\begin{figure}
	\includegraphics[scale=1]{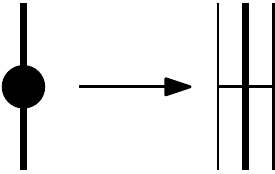}
	\caption{Cross cuts on an abstract link diagram inherited from cut loci.}
	\label{Fig:ALDCutLoci}
\end{figure}

\begin{definition}[Definition 3.2 of \cite{Carter2000}]
\label{Def:CKS1}
Let \( \left( F_1, D_1 \right) \) and \( \left( F_2, D_2 \right) \) be abstract link diagrams. We say that \( \left( F_1, D_1 \right) \) and \( \left( F_2, D_2 \right) \) are equivalent, denoted \( \left( F_1, D_1 \right) \leftrightsquigarrow \left( F_2, D_2 \right) \), if there exists a closed, connected, oriented surface \( F_3 \) and orientation-preserving embeddings \( f_1 : F_1 \rightarrow F_3 \), \( f_2 : F_2 \rightarrow F_3 \) such that \( f_1 ( D_1 ) \) and \( f_2 ( D_2 ) \) are related by Reidemeister moves on \( F_3 \). We say that two abstract link diagrams \( \left( F, D \right) \) and \( \left( F^{\prime}, D^{\prime} \right) \) are \emph{stably equivalent} if there is a chain of equivalences
\begin{equation*}
\left( F, D \right) = \left( F_0, D_0 \right) \leftrightsquigarrow \left( F_1, D_1 \right) \leftrightsquigarrow \dots \leftrightsquigarrow \left( F_n, D_n \right) = \left( F^{\prime}, D^{\prime} \right)
\end{equation*}
for \( n \in \mathbb{N} \).\CloseDef
\end{definition}

Stable equivalence classes of abstract link diagrams are in bijective correspondence to equivalence classes of virtual link diagrams \cite{Kamada2000}.

\begin{definition}
\label{Def:abstractsmoothing}
	A \emph{smoothing} of an abstract link diagram \( \left( F , D \right) \) is a diagram formed by smoothing the crossings of \( D \) into either their \( 0 \)- or \( 1 \)-resolution on \( F \). The result is a collection of disjoint copies of \( S^1 \) on the surface \( F \). A copy of \( S^1 \) is called a \emph{cycle}.\CloseDef
\end{definition}

The diagram-level canonical generators of the Lee complex given in \cite{Dye2014} are smoothings of abstract link diagrams with extra information added. This extra information keeps track of the source-sink structure of the virtual knot. The information is in the form of \textit{cross cuts} which are added in the following way: before beginning the procedure described above mark the virtual knot diagram with cut loci as inherited from the source-sink orientation and preserve them on the abstract link diagram. Replace each cut locus with a cross cut which bisects the surface as shown in \Cref{Fig:ALDCutLoci}. Henceforth by \textit{abstract link diagram} we mean an \textit{abstract link diagram with cross cuts}.

Using the source-sink decoration we add yet more information to abstract link diagrams in the form of a \textit{checkerboard colouring}:

\begin{definition}
\label{Def:checkerboard}
	From an abstract link diagram \( \left( F , D \right) \) form its associated \emph{checkerboard coloured} abstract link diagram from the surface and curve pair \( \left( F , S ( D ) \right) \) (where \( S ( D ) \) denotes the source-sink diagram formed by replacing each crossing by the source-sink decoration) by colouring the surface \( F \) using the recipe given in \Cref{Fig:checkerboard} and \Cref{Fig:checkerboardcutloci}.
	\begin{figure}
		\includegraphics[scale=1]{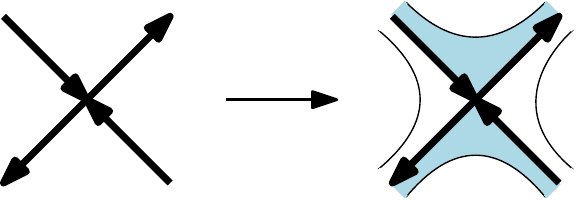}
		\caption{Checkerboard colouring at a crossing.}
		\label{Fig:checkerboard}
	\end{figure}
	\begin{figure}
		\includegraphics[scale=1]{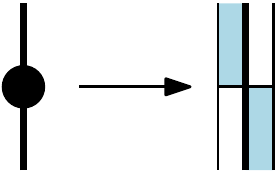}
		\caption{Checkerboard colouring at a cut locus.}
		\label{Fig:checkerboardcutloci}
	\end{figure}
	Notice that \Cref{Fig:checkerboard} allows us to induce a checkerboard colouring of smoothings of abstract link diagrams by simply joining the shaded or unshaded areas produced by smoothing the crossing.\CloseDef
\end{definition}

From checkerboard coloured smoothings of abstract link diagrams we are able to produce the tools used by Dye, Kaestner, and Kauffman to prove theorems analogous to those in \cite{Bar-Natan2006}. Henceforth we set \( \mathcal{R} = \mathbb{Q} \) and \( t = -1 \).

\begin{definition}
\label{Def:redgreen}
	 Let \( \lbrace r, g \rbrace \) be the basis for \( \mathcal{A} \) where
	 \begin{equation*}
	 \begin{aligned}
	 \text{``red"} &= r = \dfrac{1 + X}{2} \\
	 \text{``green"} &= g = \dfrac{1 - X}{2}.
	 \end{aligned}
	 \end{equation*}
	 On the level of diagrams, arcs of a smoothing are coloured red or green to denote which generator they are labelled with.\CloseDef
\end{definition}

The properties of \( r \) and \( g \) are listed in Lemma \( 4.1 \) of \cite{Dye2014}. The most important for our purposes is that \( r \) and \( g \) are conjugates with respect to the barring operator. That is
\begin{equation*}
\overline{r} = g ~ \text{and} ~ \overline{g} = r.
\end{equation*}

\begin{definition}[Analogue of Definition \(1.1\) of \cite{Bar-Natan2006}]
	\label{Def:alternatelycoloured}
	An \emph{alternately coloured smoothing of an abstract link diagram} is a smoothing for which the arcs have been coloured either red or green such that the arcs passing through each crossing neighbourhood are coloured different colours. At a cut locus the colouring of an arc switches.\CloseDef
\end{definition}

Using alternately coloured smoothings the following theorems are stated and proved:

\begin{theorem}[Theorem 4.2 of \cite{Dye2014}]
	\label{Thm:Dye1}
	Within the Karoubi envelope the Lee complex of a virtual link \( K \) is homotopy equivalent to a complex with one generator for each alternately coloured smoothing of \( K \) on an abstract link diagram with cross cuts and with vanishing differentials.
\end{theorem}

\begin{theorem}[Theorem 4.3 of \cite{Dye2014}]
	\label{Thm:Dye2}
	A virtual link \( K \) with \( | K | \) components has exactly \( 2^{|K|} \) alternately coloured smoothings on an abstract link diagram with cross cuts. These smoothings are in bijective correspondence with the \( 2^{|K|} \) orientations of \( K \).
\end{theorem}

In \Cref{Sec:cangen} we describe the bijective correspondence of \Cref{Thm:Dye2}, but we conclude this section by stating the definition of the virtual Rasmussen invariant and its properties.

\begin{definition}
	\label{Def:virtualRasmussen}
	Let \( K \) be a virtual knot diagram, \( \vckh ' ( K )\) and \( \vkh ' ( K ) \) the associated Lee complex and Lee homology, respectively. Let \( s \) be the grading on \( \vkh ' ( K ) \) induced by \( j \) on \( \vckh ' ( K ) \). Define
	\begin{equation*}
	\begin{aligned}
	s_{min} ( K ) &= \min \lbrace s ( x ) | x \in \vkh ' ( K ),~ x \neq 0 \rbrace \\
	s_{max} ( K ) &= \max \lbrace s ( x ) | x \in \vkh ' ( K ),~ x \neq 0 \rbrace.
	\end{aligned}
	\end{equation*}
	The virtual Rasmussen invariant of \( K \) is
	\begin{equation*}
	s ( K ) = \frac{1}{2} \left( s_{max} + s_{min} \right).
	\end{equation*}\CloseDef
\end{definition}

\begin{proposition}[Parts of Proposition \( 6.5 \) and Theorem \( 5.6 \) of \cite{Dye2014}]
	\label{Prop:Rasproperties}
		The virtual Rasmussen invariant satisfies the following
		\begin{enumerate}
			\item \( s ( K ) = s_{max} - 1 = s_{min} + 1 \).
			\item \( s ( \overline{K} ) = - s ( K ) \), for \( \overline{K} \) the mirror image of \( K \): the diagram formed by switching all positive classical crossings to negative classical crossings and vice versa.
			\item \( | s ( K ) | \leq 2 g^\ast ( K ) \), where \( g^\ast ( K ) \) denotes the slice genus of \( K \).
		\end{enumerate}
	
\end{proposition}
Notice that the virtual Rasmussen invariant lacks the out-of-the-box additivity of its classical counterpart (a consequence of the ill-defined nature of the connect sum operation on virtual knots). In \Cref{Sec:additive} we show, however, that the virtual \( s \) invariant is indeed additive.

\subsection{Doubled Khovanov homology}\label{Sec:doubledreview}
We review doubled Khovanov homology and the doubled Rasmussen invariant.

\subsubsection{Construction}
Doubled Khovanov homology provides an alternative extension of Khovanov homology to virtual links \cite{Rushworth2017}. The problem face is dealt with by ``doubling up'' the module assigned to a smoothing; this allows the map assigned to the single-cycle smoothing to be non-zero.

\begin{figure}
	\begin{tikzpicture}[scale=0.8]
	
	\node[] (s0)at (-5,0)  {
		\(\begin{matrix}
		0 \\
		v_+ \\
		0 \\
		v_-
		\end{matrix}\)};
	
	\node[] (s1)at (-3,0)  {
		\(\begin{matrix}
		v_+ \\
		0 \\
		v_- \\
		0
		\end{matrix}\)};
	
	\draw[->,thick] (s0)--(s1) node[above,pos=0.5]{\( \eta \)} ;
	
	\node[] (s2)at (3,0)  {
		\(\begin{matrix}
		0 \\
		\vup \\
		\vlp \\
		\vum \\
		\vlm
		\end{matrix}\)};
	
	\node[] (s3)at (5,0)  {
		\(\begin{matrix}
		\vup \\
		\vlp \\
		\vum \\
		\vlm \\
		0
		\end{matrix}\)};
	
	\draw[->,thick] (s2)--(s3) node[above,pos=0.5]{\( \eta \)} ;
	
	\end{tikzpicture}
	\caption{On the left, the complex of to the single-cycle smoothing if one assigns \( \mathcal{A} \) to a cycle. On the right, the complex of the single-cycle smoothing if one assigns \( \mathcal{A} \oplus \mathcal{A} \lbrace -1 \rbrace \) to a cycle. The generators are arranged vertically by quantum grading.}
	\label{Fig:doublingup}
\end{figure}

A schematic picture of this ``doubling up'' process is given in \Cref{Fig:doublingup}; the left hand complex depicts the situation when the module \( \mathcal{A} \) is assigned to a cycle within a smoothing. One sees immediately that the \( \eta \) map must be zero if it is to be degree-preserving. This is path followed by Manturov and Dye et al, and outlined in the previous section. The right hand complex, however, depicts the situation arrived at if one assigns the module \( \mathcal{A} \oplus \mathcal{A} \lbrace -1 \rbrace \) to a cyle, where \( \mathcal{A} = \langle \vup, \vum \rangle \) and \( \mathcal{A} \lbrace -1 \rbrace = \langle \vlp, \vlm \rangle \) (the superscripts are u for ``upper'' and l for ``lower''). This allows for \( \eta \) to be non-zero and degree preserving.

Given a virtual link diagram, \( D \), the complex \( \cdkh ( D ) \) is formed in the usual way: form the cube of resolutions of \( D \), then assign modules to the vertices and maps to the edges. The module assigned to a smoothing of \( j \) cycles is \( \mathcal{A}^{\otimes j} \oplus  \mathcal{A}^{\otimes j} \lbrace -1 \rbrace \). The maps constituting the differential are as follows. The \(m\) and \( \Delta \) maps are effectively unchanged:
	\begin{equation*}
		\label{Eq:diffcomp}
		\begin{aligned}
			m( v^{\text{\emph{u/l}}}_+ \otimes v^{\text{\emph{u/l}}}_+ ) & = v^{\text{\emph{u/l}}}_+ \qquad &\Delta ( v^{\text{\emph{u/l}}}_+ ) & = v^{\text{\emph{u/l}}}_+ \otimes v^{\text{\emph{u/l}}}_-  + v^{\text{\emph{u/l}}}_- \otimes v^{\text{\emph{u/l}}}_+  \\
			m( v^{\text{\emph{u/l}}}_+ \otimes v^{\text{\emph{u/l}}}_- ) &= m( v^{\text{\emph{u/l}}}_- \otimes v^{\text{\emph{u/l}}}_+ ) = v^{\text{\emph{u/l}}}_- \qquad &\Delta  ( v^{\text{\emph{u/l}}}_- ) & = v^{\text{\emph{u/l}}}_- \otimes v^{\text{\emph{u/l}}}_- \\
			m(v^{\text{\emph{u/l}}}_- \otimes v^{\text{\emph{u/l}}}_- ) & = 0 & & 
		\end{aligned}
	\end{equation*}
(notice that they do not map between the upper and lower summands). The \( \eta \) map associated to the single cycle smoothing as in \Cref{Fig:121} is given by
	\begin{equation*}
		\label{Eq:etamap}
		\begin{aligned}
			\eta ( v^{\text{\emph{u}}}_+ ) & = v^{\text{\emph{l}}}_+ \qquad & \eta ( v^{\text{\emph{l}}}_+ ) & = 2 v^{\text{\emph{u}}}_- \\
			\eta ( v^{\text{\emph{u}}}_- ) & = v^{\text{\emph{l}}}_- \qquad & \eta ( v^{\text{\emph{l}}}_- ) & = 0.
		\end{aligned}
	\end{equation*}
We denote by \( \dkh ( L ) \) the homology of the complex \( \cdkh ( D ) \), where \( L \) is the link represented by \( D \). We refer the reader to \cite{Rushworth2017}.

\subsubsection{The doubled Rasmussen invariant}
As in classical Khovanov and MDKK theories there is a perturbation of doubled Khovanov homology produced by adding a term of degree \(+4\) to the differential. As in the other cases, this perturbation allows the definition of a concordance invariant. In this section we give the essentials we require for \Cref{Subsec:doubledbounds}, for full details we refer the reader to \cite[Section \(4\)]{Rushworth2017}.

Given a virtual link diagram, \( D \), let \( \cdkh ' ( D ) \) denote the complex with the chain spaces of \( \cdkh ( D ) \) but with altered differential. The homology of \( \cdkh ' ( D ) \) is an invariant of the link represented by \( D \), and is denoted \( \dkh ' ( L ) \) (where \( L \) is the link represented by \( D \)). The complex \( \cdkh ' ( D ) \) is refered to as the doubled Lee complex, and the homology as the doubled Lee homology.

The rank of doubled Lee homology of a link depends on the number of alternately coloured smoothings the link possesses - here we mean the usual notion of alternately coloured smoothing, rather than the augmented notion of alternately coloured smoothings on abstract link diagrams used in \Cref{Sec:DKKreview}. Unlike classical links, virtual links may posesses no alternately coloured smoothings. (In fact, one of the purposes of the extra decoration applied to diagrams in the construction of MDKK homology is to ensure that the oriented smoothing of the augmented diagrams is always alternately colourable.)

\begin{theorem}[Theorem \(3.5\) of \cite{Rushworth2017}]
	\label{Thm:leerank}
	Given a virtual link \( L \)
	\begin{equation*}
	\text{rank} \left( \dkh' ( L ) \right) = 2 \left| \left\{ \text{alternately coloured smoothings of}~ L \right\} \right |.
	\end{equation*}
\end{theorem}

Further, given a diagram \( D \) of a virtual link \( L \), each alternately coloured smoothing, \( \mathscr{S} \), (if any exist) defines two generators of \( \dkh ' ( L ) \), denoted \( \sg^\text{u} \) and \( \sg^\text{l} \) and known as an \emph{alternately coloured generators}.

A virtual knot has two alternately coloured smoothings \cite[Theorem \(3.12\)]{Rushworth2017} so that its doubled Lee homology is of rank \( 4 \). The four generators of the homology lie in a single homological degree, and the quantum grading of any one of them determines that of the others \cite[Lemma \(4.2\)]{Rushworth2017}. Thus, for a virtual knot, \( K \), the information contained in \( \dkh ' ( K ) \) is equivalent to a pair of integers.

\begin{definition}[Definition \( 4.5 \) of \cite{Rushworth2017}]
	For a virtual knot \( K \) the \emph{doubled Rasmussen invariant} is denoted \( \mathbbm{s} ( K ) = ( s_1 ( K ), s_2 ( K ) ) \in \mathbb{Z} \times \mathbb{Z} \), where \( s_1 ( K ) \) is equivalent to the highest non-trivial quantum degree of \( \dkh ' ( K ) \), and \( s_2 ( K ) \) is the single non-trivial homological degree of \( \dkh ' ( K ) \).
\end{definition}

The component \( s_2 ( K ) \) is easy to compute from any diagram, \( D \), of \( K \): it is the height of the alternately coloured smoothings of \( D \). It is also equal to the odd writhe of \( K \) (see \cite[Proposition \( 4.11 \)]{Rushworth2017}).

\subsection{Even knots}\label{Subsec:evenknots}
To conclude this section we give a class of virtual knots for which the two extensions of the Rasmussen invariant are equal.

Recall the definition of an even virtual knot given in \Cref{Subsec:results}; here prove a fact about the cube of resolutions associated to even virtual knot diagrams.

\begin{proposition}
	\label{Prop:noeta}
	Let \( D \) be an even virtual knot diagram. Then \( vCKh ( D ) \) and \( CDKh ( D ) \) contain no \( \eta \) maps.
\end{proposition}

\begin{proof}
	As \( D \) is even it possesses a global source-sink orientation i.e.\ applying the source-sink decoration does not yield any cut loci. (In fact, possessing a global source-sink structure is equivalent to being even, but here we only need one direction.) To see this orient \( D \) with either of it's orientations (the usual notion of orientation, not source sink), and consider leaving a classical crossing of \( D \) and returning to the arc proscribed by the usual orientation. One sees from \Cref{Fig:sourcesink} that passing through a classical crossing reverses the source-sink orientation. As all classical crossings of \( D \) are even, one passes through an even number of crossings between leaving and returning at the proscribed arc. Thus the source-sink orientation has been reversed an even number of times, yielding no overall change. This argument can be applied at every crossing to show that \( D \) has a global source-sink orientation.
	
	Next, notice that every smoothing of \( D \) inherits an orientation from the global source-sink orientation of \( D \): looking again at \Cref{Fig:sourcesink} one sees that both resolutions of the classical crossing inherit an orientation from the source-sink decoration. That the orientation inherited is consistent between distinct classical crossings of \( D \) follows from that fact that \( D \) has no cut loci.
	
	Finally, we notice that if every smoothing of \( D \) inherits a coherent orientation from the global source-sink orientation of \( D \) then every cycle within a smoothing must look as in the left or center of \Cref{Fig:smoothingorientation}, as the configuration on the right is prohibited for reasons of (source-sink) orientation. But we see that the configurations on the left and center correspond to either a merge or a split, while the configuration on the right corresponds to the single-cycle smoothing. Thus no single-cycle smoothings can occur in the cube of resolutions of \( D \) and we arrive at the desired result.
\end{proof}

\begin{figure}
	\includegraphics[scale=0.75]{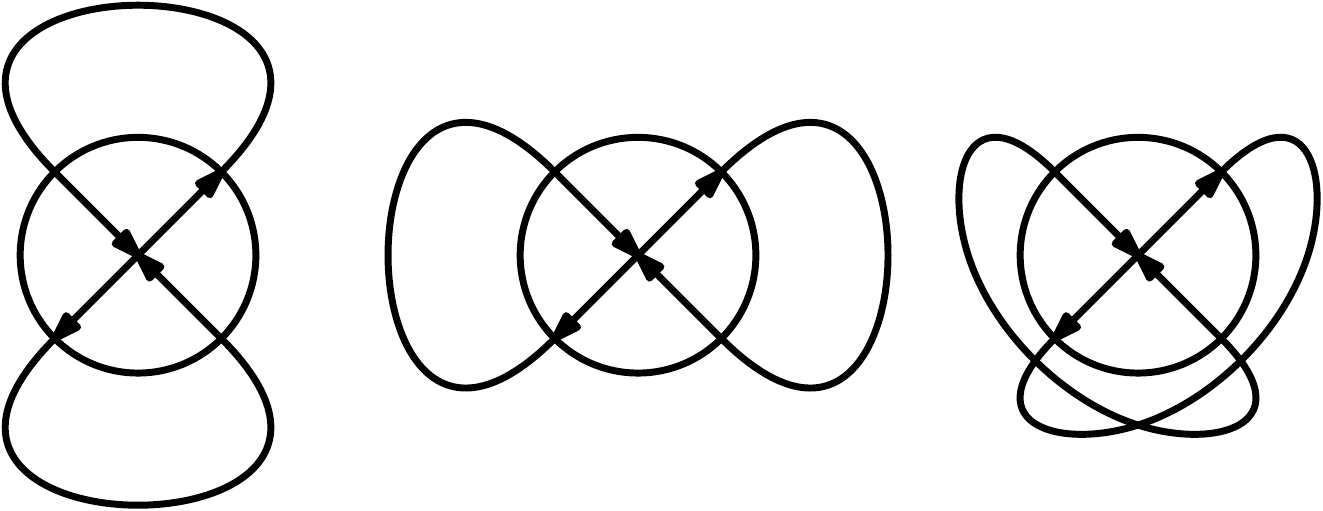}
	\caption{Configurations of cycles within a smoothing of a diagram possessing a global source-sink orientation. Two possible configurations are at the left and center, while an impossible configuration is at the right.}
	\label{Fig:smoothingorientation}
\end{figure}

\begin{corollary}
	\label{Cor:evenknots}
	Let \( K \) be an even virtual knot. Then \( DKh ( K ) = vKh ( K ) \oplus vKh ( K ) \lbrace -1 \rbrace \) so that \( s ( K ) = s_1 ( K ) \).
\end{corollary}

\begin{proof}
	Let \( D \) be an even diagram of \( K \). Then both \( vCKh ( D ) \) and \( CDKh ( D ) \) contain no \( \eta \) maps by \Cref{Prop:noeta}. As \( m \) and \( \Delta \) do not map between the shifted and unshifted summands of \( CDKh ( D ) \), the complex splits as the direct sum \( CDKh ( D ) = vCKh ( D ) \oplus vCKh ( D ) \lbrace -1 \rbrace \).
\end{proof}

\section{Chain-level generators of \( vKh ' \)}\label{Sec:cangen}
\begin{figure}
	\includegraphics[scale=0.4]{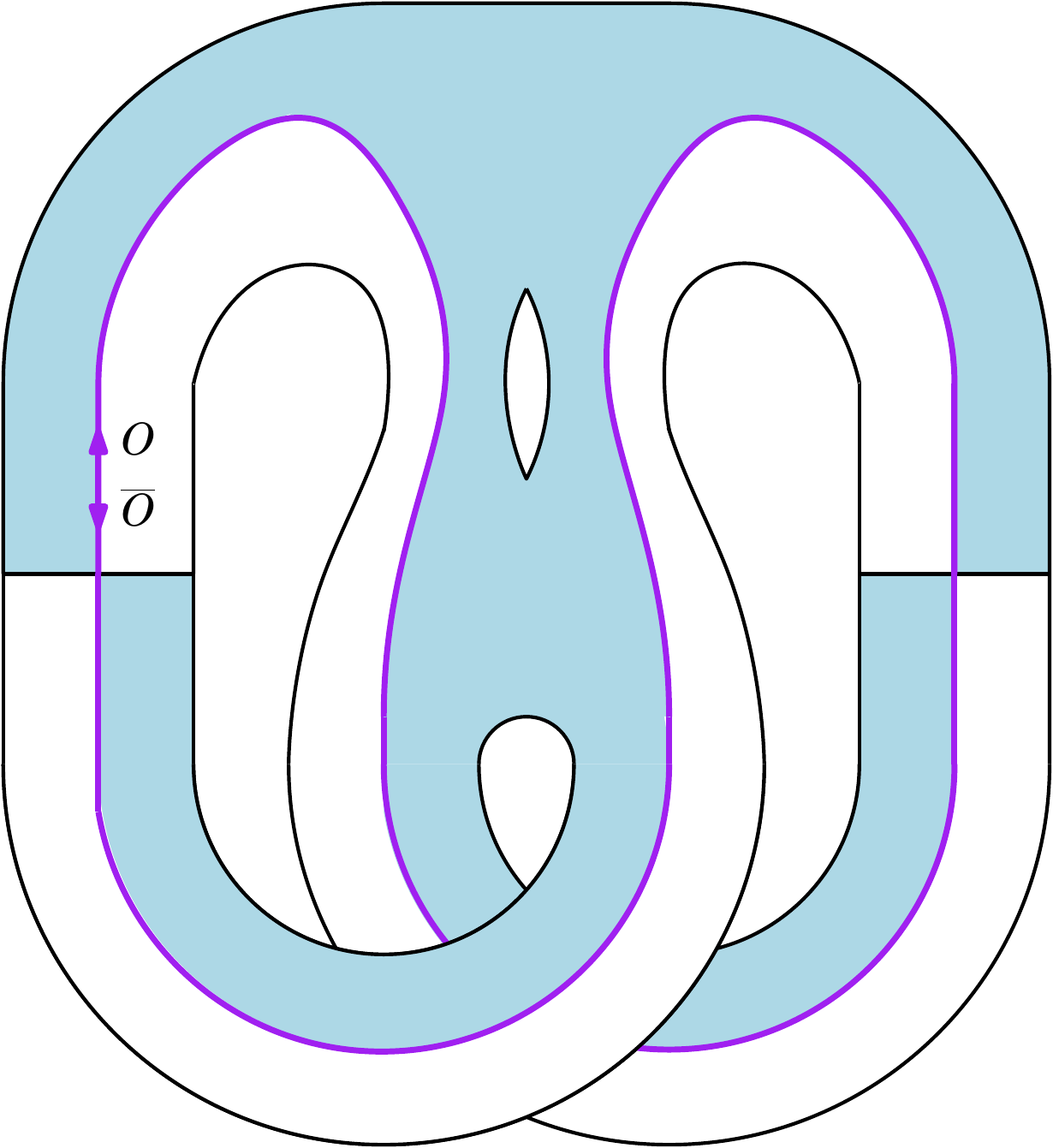}
	\caption{A checkboard coloured abstract link diagram corresponding to the virtual knot diagram of \Cref{Fig:vknot}, with orientations \( o \) and \( \overline{o} \).}
	\label{Fig:cb_K}
\end{figure}
\begin{figure}
	\includegraphics[scale=0.4]{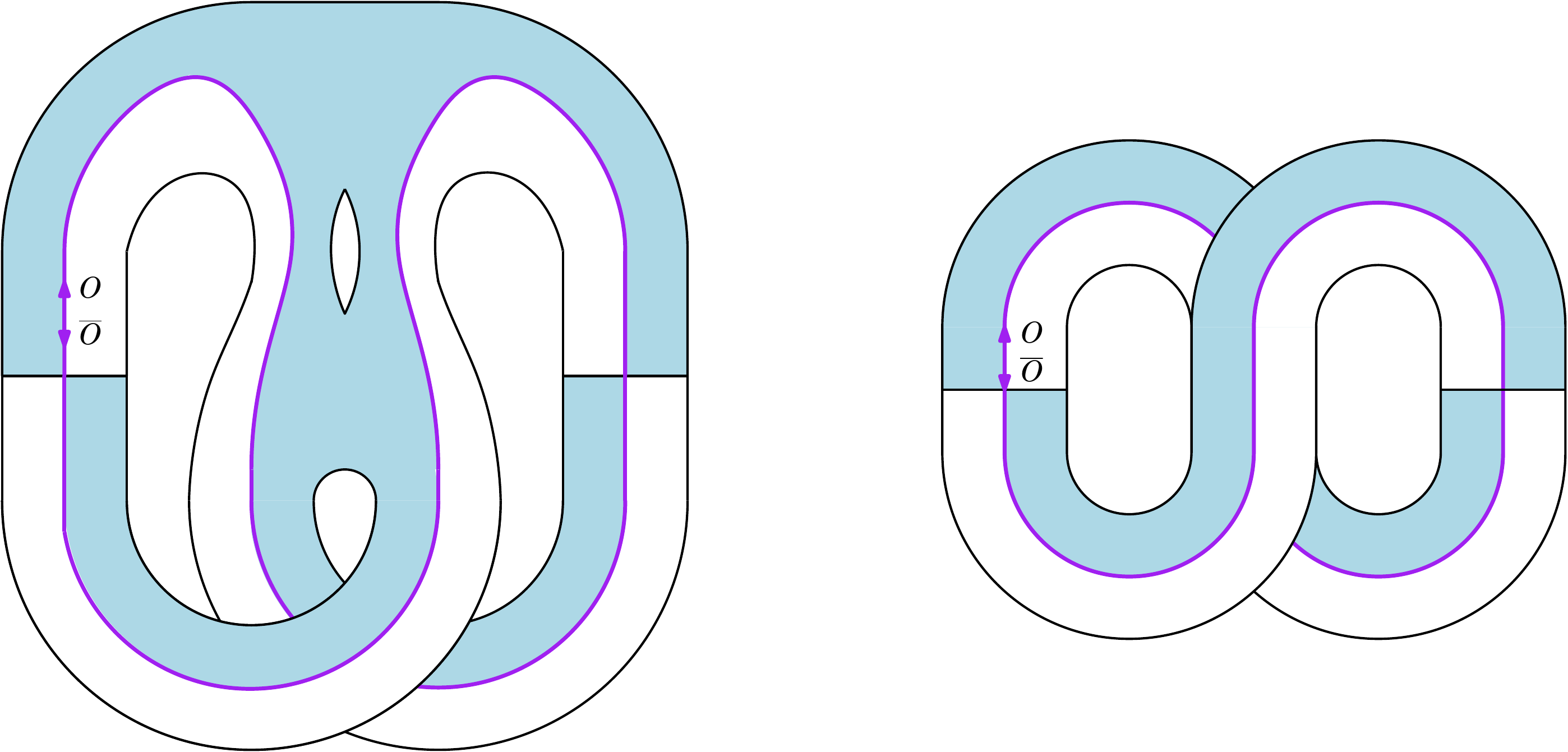}
	\caption{Two representatives of the stable equivalence class of smoothings of the checkboard coloured abstract link diagram depicted in \Cref{Fig:cb_K}, with orientations  \( o \) and \( \overline{o} \).}
	\label{Fig:cb_smooth}
\end{figure}
\begin{figure}
	\centering
	\begin{subfigure}[b]{0.4\textwidth}
		\begin{center}
			\includegraphics[scale=0.35]{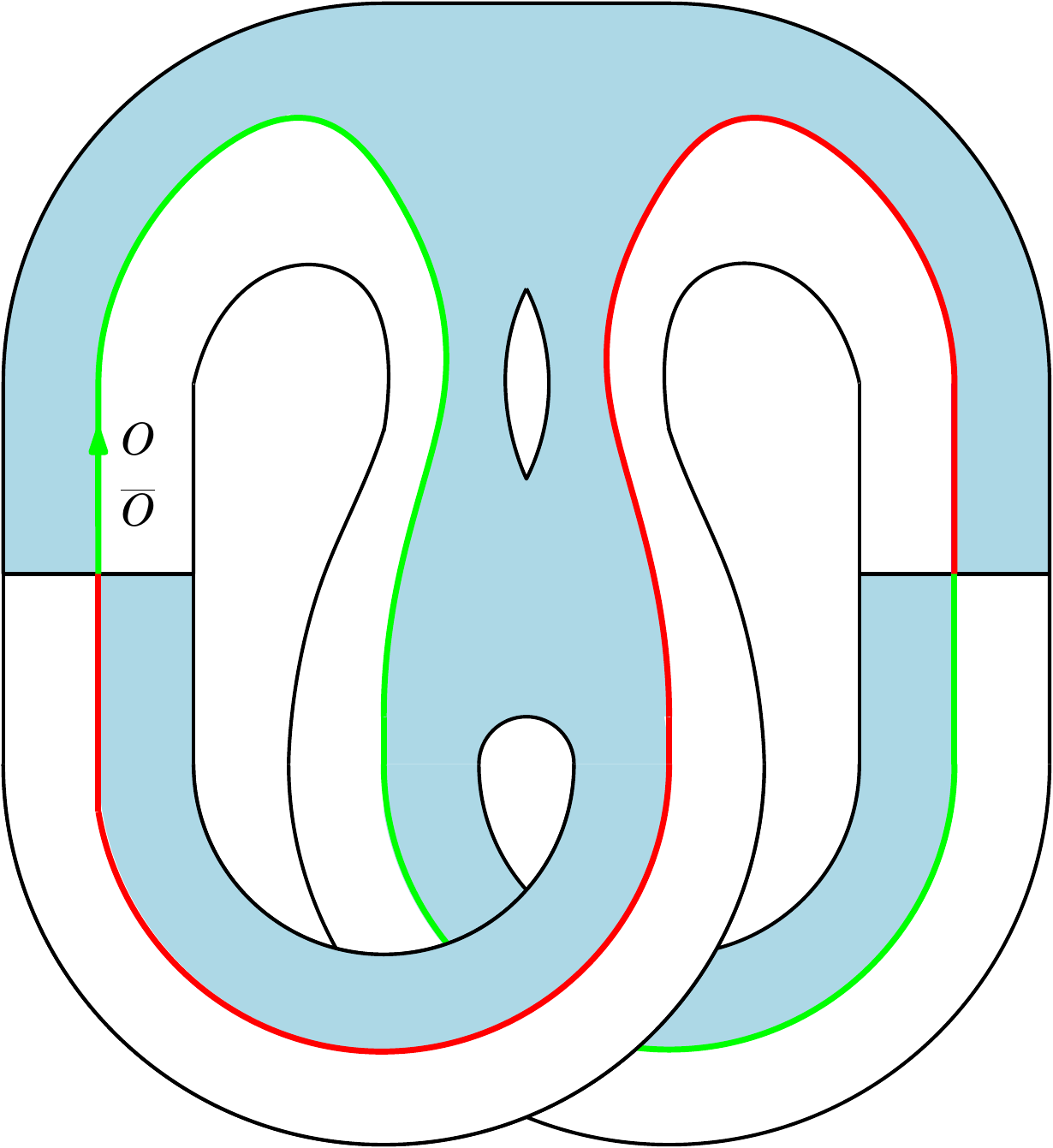}
			\caption{The alternately coloured smoothing associated to orientation \( o \).}
			\label{Fig:ac_smoothing_o}
		\end{center}
	\end{subfigure}
	~
	\begin{subfigure}[b]{0.4\textwidth}
		\begin{center}
			\includegraphics[scale=0.35]{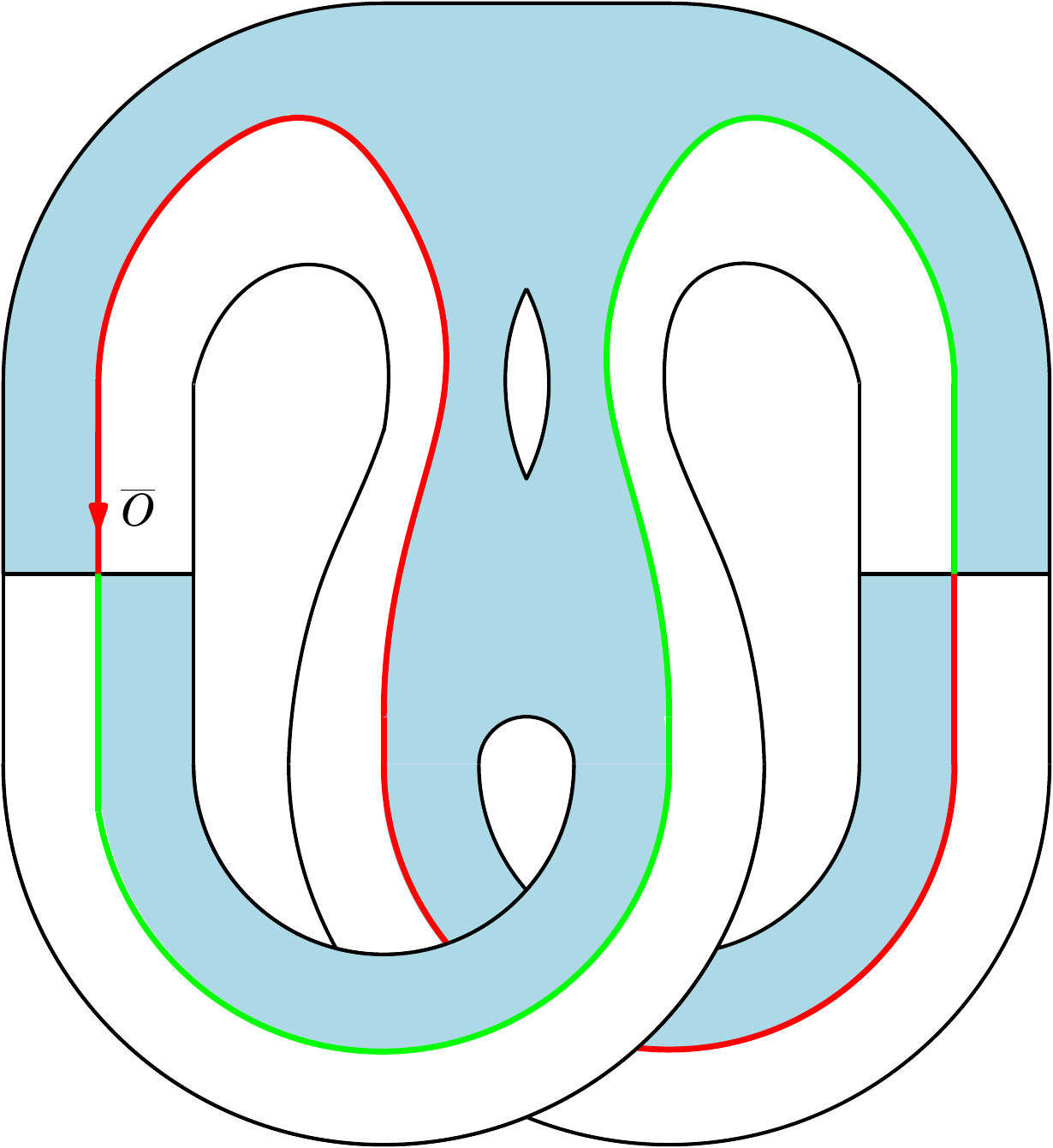}
			\caption{The alternately coloured smoothing associated to orientation \( \overline{o} \).}
			\label{Fig:ac_smoothing_obar}
		\end{center}
	\end{subfigure}
	\caption{The alternately coloured smoothings on abstract link diagrams corresponding to the two possible orientations of the virtual knot diagram.}\label{Fig:alternatesmoothings}
\end{figure}

\begin{figure}
	\centering
	\begin{subfigure}[b]{0.4\textwidth}
		\begin{center}
			\includegraphics[scale=0.4]{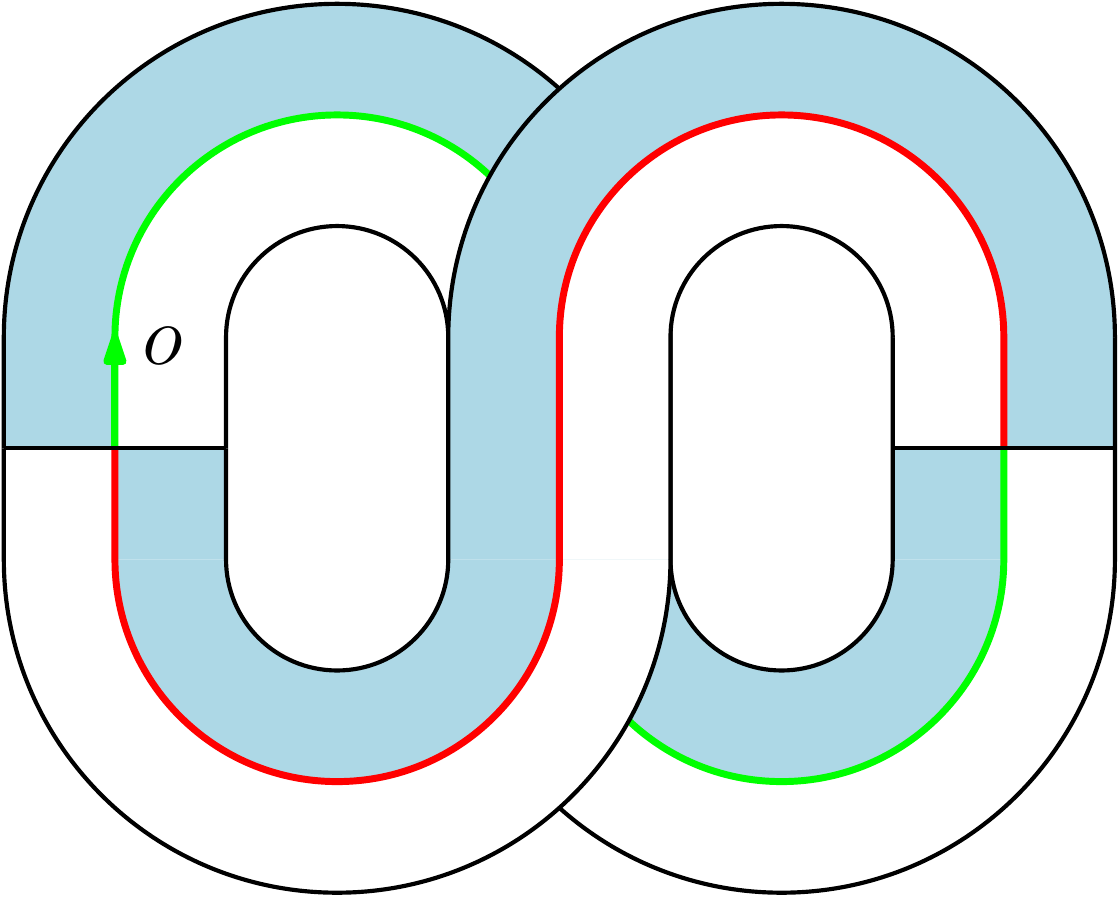}
			\caption{A smoothing stably equivalent to that of \Cref{Fig:ac_smoothing_o}.}
			\label{Fig:stable_smoothing_o}
		\end{center}
	\end{subfigure}
	~
	\begin{subfigure}[b]{0.4\textwidth}
		\begin{center}
			\includegraphics[scale=0.4]{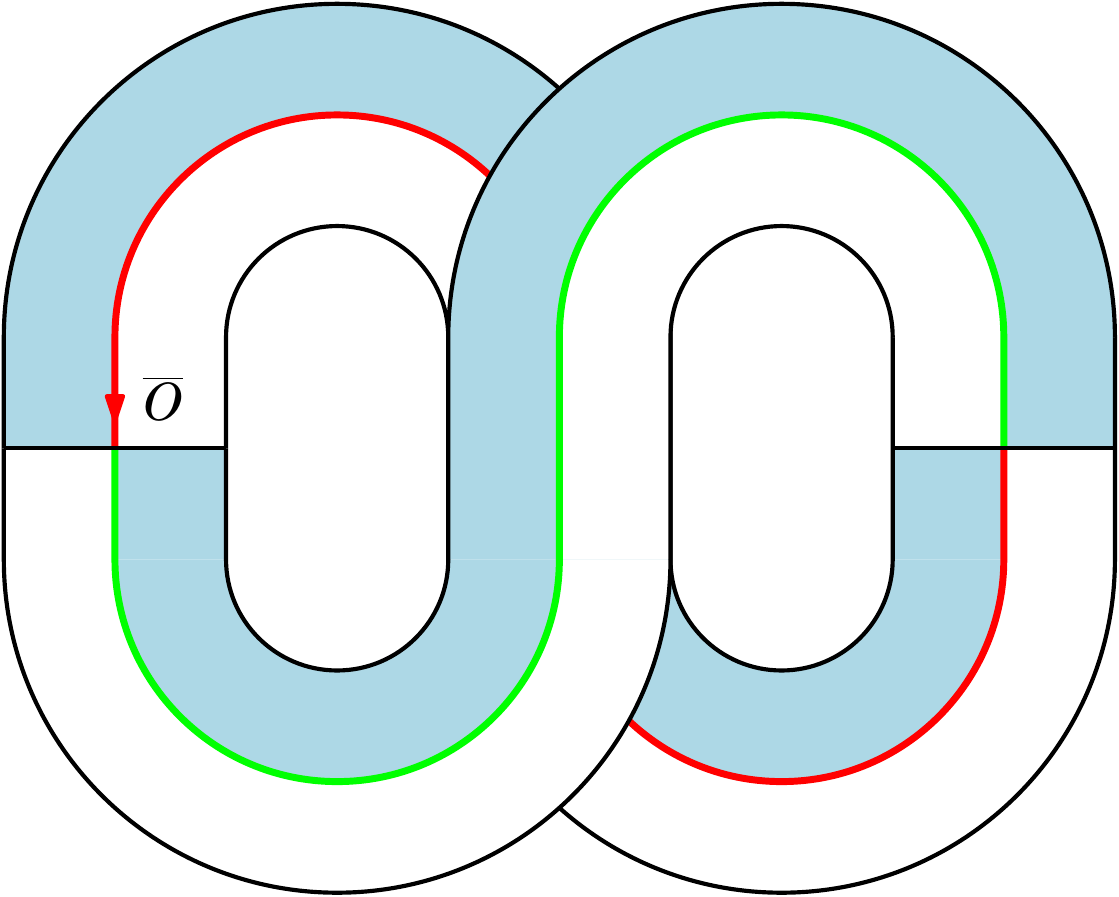}
			\caption{A smoothing stably equivalent to that of \Cref{Fig:ac_smoothing_obar}.}
			\label{Fig:stable_smoothing_obar}
		\end{center}
	\end{subfigure}
	\caption{Alternately coloured smoothings stably equivalent to those of \Cref{Fig:alternatesmoothings}.}\label{Fig:stablesmoothings}
\end{figure}

In \cite{Dye2014} canonical generators are produced at a diagrammatic level i.e.\ they are alternately coloured smoothings of (checkerboard-coloured) abstract link diagrams. These generators are sufficient to prove Theorems \ref{Thm:Dye1} and \ref{Thm:Dye2}. Below, we give a method to produce the corresponding chain-level generators of \( \vkh' ( K )\). Before doing so, however, it is instructive to recall the bijection of \Cref{Thm:Dye2} between orientations of a virtual link and alternately coloured smoothings of the associated abstract link diagram as given in \cite{Dye2014}. We use \Cref{Fig:vknot} as an example.
\begin{enumerate}[(i)]
	\item Given a virtual link diagram \( D \) construct the checkerboard coloured abstract link diagram as in \Cref{Def:checkerboard}. Note that for a virtual knot the checkerboard colouring is independent of the orientation, a consequence of the invariance of the source-sink decoration under \( 180^{\circ} \) rotations. See \Cref{Fig:cb_K}.
	\item For a given orientation \( o \) of \( D \) form the corresponding oriented smoothing on the checkerboard coloured abstract link diagram as in \Cref{Def:abstractsmoothing}. See \Cref{Fig:cb_smooth}.
	\item\label{part3} Place a clockwise orientation on the shaded regions of the oriented smoothing, which in turn induces a new orientation on the arcs of the smoothing. On each arc compare this orientation to that induced by \( o \). If these two orientations agree colour the arc red, if they disagree colour the arc green (as in \Cref{Def:redgreen}). See \Cref{Fig:alternatesmoothings}.
\end{enumerate}
At this stage we have produced alternately coloured smoothings on abstract link diagrams as in \Cref{Def:alternatelycoloured}. We need a way of reading off from these diagrams elements of \( {\vckh '}_0 ( K ) \) (as the oriented resolution is always at height \( 0 \)), which will be the chain-level canonical generators of \( \vkh ' ( K ) \). We are unable to do so at this point as the cycles of the alternately coloured smoothings possess more than one colour. We now describe a process by which single coloured smoothings can be produced, and hence chain-level generators of \( \vkh ' ( K ) \).

Firstly, we utilise the stable equivalence relation given in \Cref{Def:CKS1} to work with alternately coloured smoothings of abstract link diagrams for which the surface deformation retracts onto the curve of the smoothing, for example the abstract link diagrams given in \Cref{Fig:stablesmoothings}. We can always do this as the curve of the smoothing is simply a disjoint union of copies of \( S^1 \). Note that the resulting smoothing (of a checkerboard coloured abstract link diagram) may not be connected.

Next, we interpret the cross cuts as half-twists with the parity of the twist ignored. That is
\[
	\raisebox{-22.5pt}{\includegraphics[scale=1]{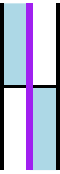}} \quad = \raisebox{-26.5pt}{\includegraphics[scale=1]{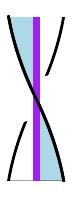}} ~\text{or equivalently}~ \raisebox{-26.5pt}{\includegraphics[scale=1]{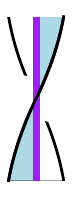}}.
\]
The author learnt of this interpretation in the talks of Dye and of Kaestner during Special Session 35, ``Low Dimensional Topology and Its Relationships with Physics'', of the 2015 AMS/EMS/SPM Joint Meeting.

Replacing cross cuts with appropriate half-twists we are able to view the surface of the smoothing (of a checkerboard coloured abstract link diagram) as a two-sided surface such that the curve of the smoothing appears on both sides. That cross cuts always come in pairs ensures that the surface has two sides. Importantly, on each side of the surface the curve of the smoothing is coloured exactly one colour. This is because passing a cross cut causes the arc to change to change colour (c.f.\ \Cref{Def:alternatelycoloured}), and to pass a cut locus is to pass onto the other side of the surface. (From this one can see that passing a cut locus, or equivalently moving on to the other side of the surface, is replicated in  \( \mathcal{A} \) by applying the barring operator.)

In summary, we view alternately coloured smoothings (of checkerboard coloured abstract link diagrams) such as those in \Cref{Fig:stablesmoothings} as two sided surfaces such that the curve of the smoothing is coloured exactly one colour on each side. At this point it is clear that in order to read off generators of \( {\vckh '}_0 ( K ) \) from such alternately coloured smoothings we must make a choice of side (or sides, if the surface of the smoothing is disconnected) of the surface to read. Further, we must also ensure that this choice is the same for both the alternately coloured smoothings associated to \( o \) and \( \overline{o} \). We must have this as they are both coloured versions of the same smoothing of an abstract link diagram (the oriented smoothing) c.f. the left hand smoothing of \Cref{Fig:cb_smooth} with \Cref{Fig:alternatesmoothings}. In effect we are making the choice on this uncoloured smoothing, which the alternately coloured smoothings then inherit.

With all this in mind, let us make a choice: given a virtual knot diagram \( K \) with orientations \( o \) and \( \overline{o} \), let \( A \) denote the oriented smoothing of the checkerboard coloured abstract link diagram associated to \( K \). On \( A \) cancel an arbitrary pair of adjacent cross cuts against one another so that the strand they bound is removed. An example is given in \Cref{Fig:cancellation}. This cancellation of cross cuts is simply `flipping' the segment of the surface they bound so that the other side of the surface is shown. Continue cancelling available arbitrary pairs of cross cuts until all have been removed. In our interpretation, that the smoothing has no cross cuts means that we are looking at exactly one side of surface. Now return to part (iii) of the process given on page \pageref{part3}, and colour the cycles of the oriented smoothings associated to \( o \) and \( \overline{o} \) as dictated there. Denote by \( A_o \) and \( A_{\overline{o}} \) the resulting alternately coloured abstract link diagrams associated to \( o \) and \( \overline{o} \), respectively. That the cycles of \( A_o \) and \( A_{\overline{o}} \) are coloured with opposite colours follows from the fact that their orientations are opposite but the checkerboard colouring of \( A_o \) and \( A_{\overline{o}} \) is the same.

Examples of such single coloured smoothings are given in \Cref{Fig:cancelledo} and \Cref{Fig:cancelledobar}. In this case a choice of top and bottom is equivalent to picking either the two smoothings on the left of the Figures, or the two on the right.

\begin{figure}
	\includegraphics[scale=1]{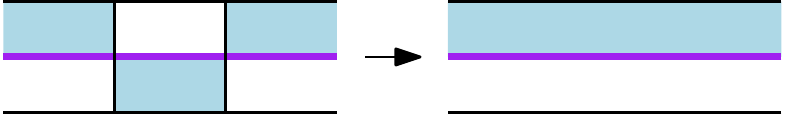}
	\caption{Removing a strand by cancelling cross-cuts.}
	\label{Fig:cancellation}
\end{figure}

\begin{figure}
	\includegraphics[scale=0.4]{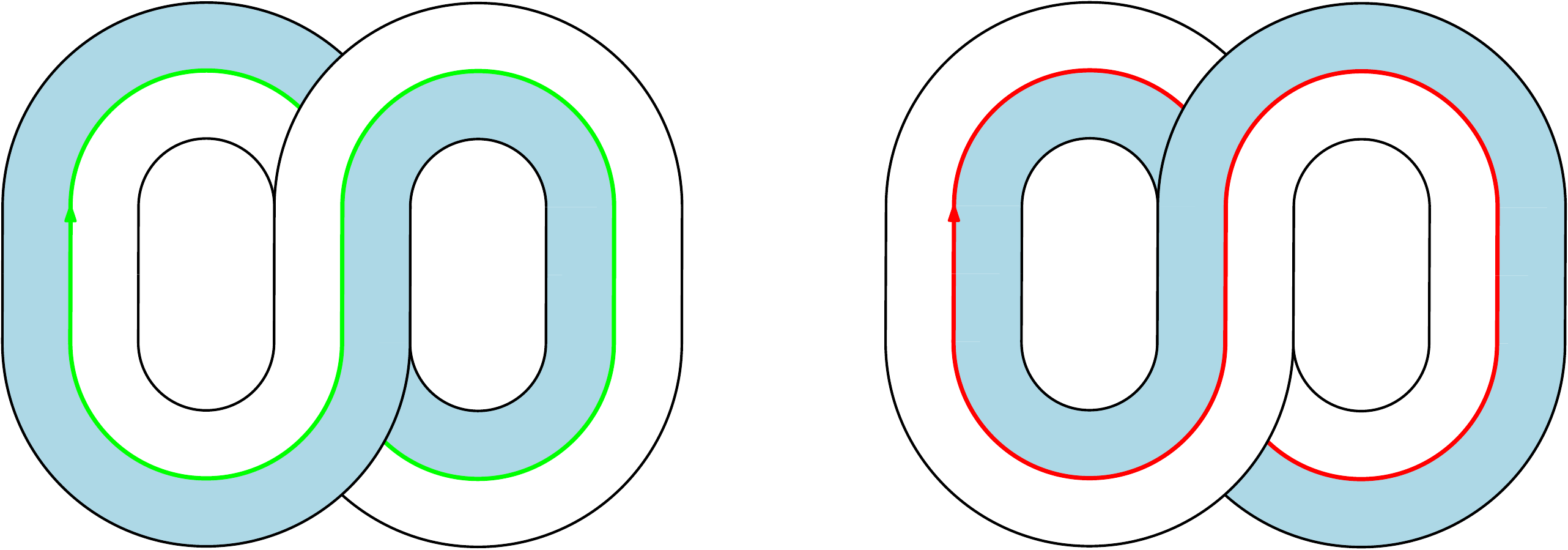}
	\caption{The possible ways to cancel the alternately coloured smoothing corresponding to orientation \( o \) of \( K \).}
	\label{Fig:cancelledo}
\end{figure}
\begin{figure}
	\includegraphics[scale=0.4]{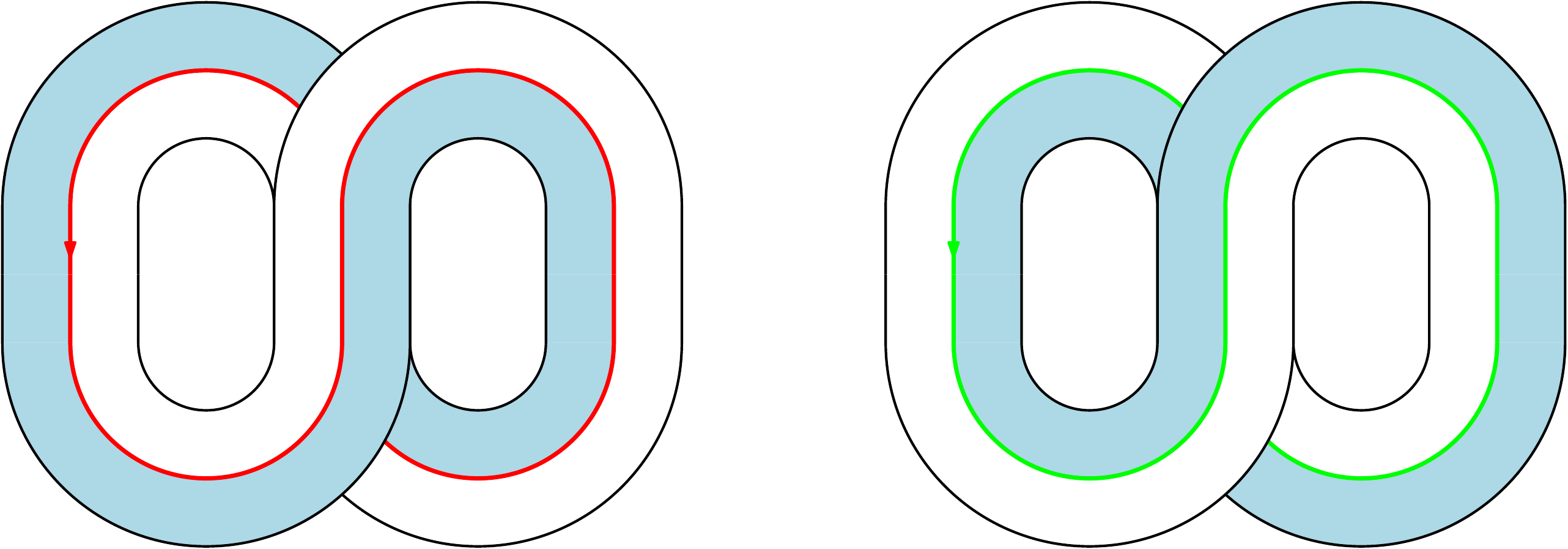}
	\caption{The possible ways to cancel the alternately coloured smoothing corresponding to orientation \( \overline{o} \) of \( K \).}
	\label{Fig:cancelledobar}
\end{figure}

After all that we are left with smoothings of abstract link diagrams the cycles of which are coloured with exactly one colour, either red or green. We form the canonical generators of \( \vkh ' ( K ) \), denoted \( \sg_o \) for \( o \) an orientation of \( K \), by taking the appropriate tensor product of \( r \) and \( g \) as dictated by the colours of the cycles. In this way we obtain two distinct algebraic generators.

We conclude by remarking that the \( s \) invariant is independent of this choice of which side of the surface to read. Making another choice results in an application of the barring operator to one or more tensor factors of \( \sg_o \) and \( \sg_{\overline{o}} \), because if a cycle is coloured green on one side of the surface it must be coloured red on the other. But conjugation does not interact with the filtration, that is
\begin{equation*}
j(\overline{r}) = j (g)~\text{and}~j(\overline{g}) = j (r).
\end{equation*}

To conclude this section we prove a Lemma analogous to Lemma \( 3.5 \) of Rasmussen \cite{Rasmussen2010} which we will use in both the following sections.

\begin{lemma}
	\label{Lem:ras}
	Let \( n \) be the number of components of \( K \).
	There is a direct sum decomposition \( \vkh'(K) \cong \vkh'_o(K) \oplus
	\vkh'_e(K) \), where \( \vkh'_o (K) \) is generated by all states with
	\( q \)-grading conguent to \( 2+n \mod{4} \), and \( \vkh'_e (K) \) is generated
	by all states with \( q \)-grading congruent to \( n \mod{4} \).
	If \( o \) is an orientation on \( K \), then \( \sg_o + \sg_{\overline{o}} \) is contained in one of the two
	summands, and \( \sg_o - \sg_{\overline{o}} \) is contained in the other.
\end{lemma}

\begin{proof}
	The first statement follows exactly as in the classical case. Regarding the second statement, following \cite{Rasmussen2010} let \( \iota : \vckh ' ( K ) \rightarrow \vckh ' ( K ) \) be the map which acts by the identity on \( \vckh'_e(K) \) and multiplication by \( -1 \) on \( \vckh'_o(K) \). We claim that \( \iota ( \sg_o ) = \pm \sg_{\overline{o}} \). To show this we define a new grading on \( \mathcal{A} \) with respect to which \( X \) has grading \( 2 \) and \( 1 \) has grading \( 4 \). We have that \( \overline{X} = - X \) and \( \overline{1} = 1 \) so that \( \overline{r} = g \) and \( \overline{g} = r \), and the map
	\begin{equation*}
	\overline{\phantom{X}}^{\otimes n} : \mathcal{A}^{\otimes n} \rightarrow \mathcal{A}^{\otimes n}
	\end{equation*}
	(which applies the barring operator to all tensor factors) acts as the identity on elements with new grading congruent to \( 0 \mod{4} \) and multiplication by \( - 1 \) on elements with new grading congruent to \( 2 \mod{4} \). The new grading differs from the \( q \)-grading by an overall shift so that
	\begin{equation*}
	\iota ( \sg_{o} ) = \pm \overline{\sg_{o}}^{\otimes n} = \pm \sg_{\overline{o}}
	\end{equation*}
	as in the classical case.
\end{proof}

A direct corollary of \Cref{Lem:ras} is that \( \sg_{o} \) is not of top filtered degree, that is:
\begin{equation}
\label{smin}
s ( \sg_{o} ) = s ( \sg_{\overline{o}} ) = s_{min} ( K ).
\end{equation}

\subsection{Additivity of the virtual Rasmussen invariant}
\label{Sec:additive}
We can use the chain-level generators of \( \vkh ' ( K ) \) to show that the virtual Rasmussen invariant is additive with respect to connect sum, confirming that the virtual invariant behaves in the same way as its classical counterpart in this respect.

The connect sum operation on virtual knots is ill-defined. That is, the result of the operation depends on both the diagrams used and the site at which the sum is conducted. As a result there exist multiple inequivalent virtual knots which can be obtained as connect sums of a fixed pair of virtual knots. By an abuse of notation we shall denote by \( K_1 \# K_2 \) any of the knots obtained as a connect sum of virtual knots \( K_1 \) and \( K_2 \).

\begin{theorem}
	\label{Thm:additive}
	For virtual knots \( K_1 \) and \( K_2 \)
	\begin{equation}
	s ( K_1 \# K_2 ) = s ( K_1 ) + s ( K_2 ).
	\end{equation}
\end{theorem}

\begin{proof}
	With the chain-level generators in place, along with \Cref{Lem:ras}, the proof follows much the same path as that in \cite{Rasmussen2010}. For all connect sums \( K_1 \# K_2 \) there exists the map
	\begin{equation*}
	\vkh ' ( K_1 \# K_2 ) \overset{\Delta '}{\longrightarrow} \vkh ' ( K_1 \sqcup K_2 ) \cong \vkh ' ( K_1 ) \otimes \vkh ' ( K_2 ).
	\end{equation*}
	It sends a canonical generator \( \sg_o \) of \( \vkh ' ( K_1 \# K_2 ) \) to a canonical generator of \( \vkh ' ( K_1 ) \otimes \vkh ' ( K_2 ) \) of the form \( \sg^1 \otimes \sg^2 \) where \( \sg^i \) is a generator of \( \vkh ' ( K_i ) \) for \( i = 1, 2 \). As in the classical case, the map is of filtered degree \( -1 \) and we obtain
	\begin{equation}
	\label{add1}
	\begin{aligned}
	s ( \sg_o ) - 1 &\leq s ( \sg^1 \otimes \sg^2 ) = s ( \sg^1 ) + s ( \sg^2 ) \\
	s_{min} ( K_1 \# K_2 ) &\leq s_{min} ( K_1 ) + s_{min} ( K_2 ), ~\text{by \Cref{smin}}.
	\end{aligned}
	\end{equation}
	From this point the proof proceeds as in that of the analogous statement in \cite{Rasmussen2010}: utilising the fact that \( s_{min} ( K ) = - s_{max} (\overline{K} ) \) we are able to obtain from \Cref{add1} that
	\begin{equation*}
	\begin{aligned}
	s_{min} ( K_1 \# K_2 ) &= s_{min} ( K_1 ) + s_{min} ( K_2 ) + 1 \\
	s_{max} ( K_1 \# K_2 ) &= s_{max} ( K_1 ) + s_{max} ( K_2 ) - 1 
	\end{aligned}
	\end{equation*}
	as required.
\end{proof}
In light of \Cref{Thm:additive} we see that the Rasmussen invariant cannot distinguish between connect sums of a fixed pair of virtual knots. In general it is not known, for \( K_1 \) and \( K_2 \) both (possibly inequivalent) connect sums of a fixed pair of virtual knots, if \( K_1 \) is concordant to \( K_2 \). It is known, however, that neither the Jones polynomial \cite{Manturov2013} nor the Rasmussen invariant can distinguish them. This leads one to posit whether Khovanov homology can; in the case of connect sums of trivial diagrams it is shown in \cite{Rushworth2017} that doubled Khovanov homology cannot.

\section{Computable bounds}\label{Sec:bounds}
In this section we extend the strong slice-Bennequin bounds to the virtual and doubled Rasmussen invariants. The bounds are constructed, and cases in which they vanish partly or wholly are described.

\subsection{The virtual Rasmussen invariant}\label{Sec:dkkformulation}
\subsubsection{Formulation}
\begin{definition}
	\label{Seifertgraph}
	Given a virtual link diagram \( D \) denote by \( O ( D ) \) the oriented smoothing of \( D \). Denote by \( T_O ( D ) \) the signed graph with a vertex for each cycle of \( O ( D ) \) and an edge for each classical crossing of \( D \), decorated with the sign of the crossing. The edge associated to a crossing is between the vertex or vertices associated to the cycles involved in the smoothing site of that crossing. The subgraph of \( T_O ( D ) \) formed by removing all the edges labelled with \( + \) (respectively \( - \)) is denoted \( T^{-}_O ( D ) \) (respectively \( T^{+}_O ( D )\)).\CloseDef
\end{definition}

The graph \( T_O ( D ) \) is often called the \emph{Seifert graph} of \( D \), but in order to avoid confusion with a graph defined in \Cref{Subsec:doubledbounds} we shall not use that term.

\begin{definition}
	\label{Def:u}
	Given a virtual knot diagram \( D \) the quantities \( U_v ( D ),~ \Delta_v ( D ) \in \mathbb{Z} \) are given by
	\begin{equation*}
	\begin{aligned}
	U_v ( D ) &= \# ~\text{vertices}~ ( T_O ( D ) ) - 2 \# ~\text{components}~ ( T^{-}_O ( D ) ) + wr ( D ) + 1 \\
	\Delta_v ( D ) &= \# ~\text{vertices}~ ( T_O ( D ) ) - \# ~\text{components}~ ( T^{+}_O ( D ) ) - \# ~\text{components}~ ( T^{-}_O ( D ) ) + 1.
	\end{aligned}
	\end{equation*}
	The quantities \( U_v ( D ) \) and \( \Delta_v ( D ) \) are dependent on the diagram \( D \) and are not invariants of the virtual knot.\CloseDef
\end{definition}

\begin{theorem}[Analogue of Theorem \( 1.2 \) of Lobb \cite{Lobb2011}]
	\label{Thm:u}
	For \( D \) a diagram of a virtual knot \( K \)
	\begin{equation*}
	s ( K ) \leq U_v ( D ).
	\end{equation*}
	Notice that the left hand side is a knot invariant whereas the right is diagram-dependent.
\end{theorem}

To prove this we require \Cref{Lem:ras}, as we have canonical generators in terms of \( r \) and \( g \) instead of \( a = 2r \) and \( b = -2g \) and the proof given in \cite{Rasmussen2010} relies on the sign of \( a \) and \( b \).

\begin{proof}
	(of \Cref{Thm:u})
	The proof is practically identical to that of the classical case given in \cite{Lobb2011}. Form the diagram \( K^{-} \) from \( K \) by smoothing all the positive classical crossings of \( K \) to their oriented resolution, and suppose that \( K^{-} \) is the disjoint union of \( l \) virtual link diagrams. Label these diagrams \( K^{-}_{1}, K^{-}_{2}, \dots, K^{-}_{l} \). Then the canonical generator \( \sg_o \) splits as a tensor product of canonical generators of \( \vkh ' ( K^{-}_{r} ) \) as
	\begin{equation*}
	\sg_o = \sg_1 \otimes \sg_2 \otimes \dots \otimes \sg_l.
	\end{equation*}
	Classically, \( \sg_r \) can either be \( \sg_{o^{\prime}} \) or \( \sg_{\overline{o^{\prime}}} \) where \( o^{\prime} \) denotes the induced orientation on \( K^{-}_{r} \), as we are possibly altering the number of cycles separating others from infinity. In the virtual case, however, \( \sg_r = \sg_{o^{\prime}} \) by construction as we use abstract link diagrams to produce the canonical generators rather than the method due to Lee.
	
	Where the proof given in \cite{Lobb2011} invokes Theorem \( 3.5 \) of \cite{Rasmussen2010} we invoke \Cref{Lem:ras} as given above.
\end{proof}

\begin{theorem}[Analogue of Theorem \( 1.10 \) of Lobb \cite{Lobb2011}]
	\label{Thm:delta}
	If \( \Delta_v ( D ) = 0 \) then \( s ( K ) = U_v ( D ) \), where \( K \) is the virtual knot represented by \( D \). In fact
	\begin{equation*}
	U_v ( D ) - 2 \Delta_v ( D ) \leq s ( K ) \leq U_v ( D ).
	\end{equation*}
\end{theorem}

The proof of \Cref{Thm:delta} is identical to that of the classical case, owing to the identical behaviour of the virtual and classical Rasmussen invariants with respect to the mirror image.

\subsubsection{The case \( \Delta_v ( D ) = 0 \)}\label{Sec:deltazero}

Cromwell defined \emph{homogeneous knots} \cite{Cromwell1989}. Here we recap his definition, which works equally well for virtual knots.

\begin{definition}
	\label{Def:cutvertex}
	A \emph{cut vertex} of a graph \( G \) is a vertex such that the graph obtained by removing the vertex along with its boundary edges has more connected components than \( G \).\CloseDef
\end{definition}

\begin{definition}
	\label{Def:block}
	A \emph{block} of a graph \( G \) is a maximal connected subgraph of \( G \) containing no cut vertices.\CloseDef
\end{definition}

\begin{definition}
	\label{Def:homogeneousgraph}
	A signed graph \( G \) is \emph{homogeneous} if every block \( B \) of \( G \) is such that all edges contained in \( B \) are decorated with the same sign.\CloseDef
\end{definition}

\begin{definition}
	\label{Def:homogeneous}
	A virtual link diagram \( K \) is \emph{homogeneous} if \( T_O ( K ) \) is homogeneous. A virtual link is homogeneous if there exists a diagram of it which is homogeneous.\CloseDef
\end{definition}

Positive and negative virtual knots are homogeneous trivially (as \( T_O ( D ) \) possesses only one kind of decoration). In the classical case alternating knots are also homogeneous \cite{Kauffman1983}. In the virtual case, however, this no longer holds. For example, the virtual knot diagram given in \Cref{Fig:3.7diagram} is alternating but not homogeneous.

\begin{figure}
	\centering
	\begin{subfigure}[b]{0.4\textwidth}
		\begin{center}
			\includegraphics[scale=0.75]{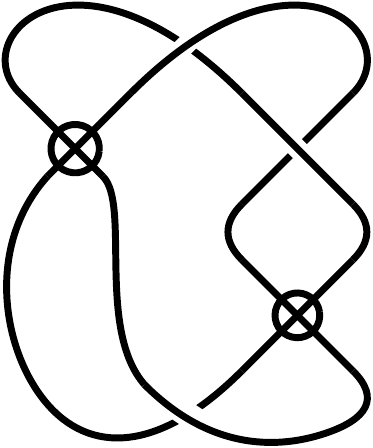}
			\caption{An alternating virtual knot diagram which is not homogeneous. It is virtual knot \( 3.7 \) in Green's table \cite{Green}.}
			\label{Fig:3.7diagram}
		\end{center}
	\end{subfigure}
	~
	\begin{subfigure}[b]{0.4\textwidth}
		\begin{center}
			\raisebox{35pt}{\includegraphics[scale=1]{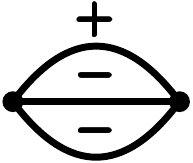}}
			\caption{The graph \( T_O ( D ) \) of virtual knot \( 3.7 \).}
			\label{Fig:3.7graph}
		\end{center}
	\end{subfigure}
	\caption{}\label{Fig:3.7}
\end{figure}

\begin{figure}
	\centering
	\begin{subfigure}[b]{0.4\textwidth}
		\begin{center}
			\includegraphics[scale=0.75]{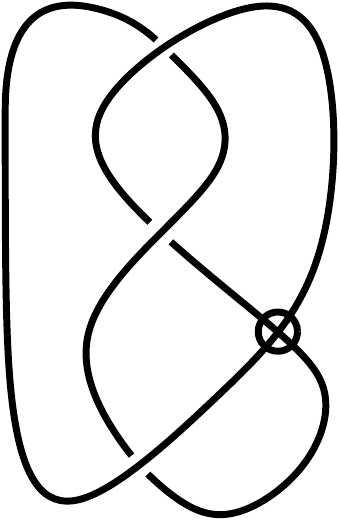}
			\caption{A diagram of virtual knot \(3.2\).}
			\label{Fig:3.2diagram}
		\end{center}
	\end{subfigure}
	~
	\begin{subfigure}[b]{0.4\textwidth}
		\begin{center}
			\raisebox{35pt}{\includegraphics[scale=1]{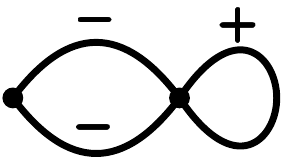}}
			\caption{The graph \( T_O ( D ) \) of virtual knot \( 3.2\).}
			\label{Fig:3.72graph}
		\end{center}
	\end{subfigure}
	\caption{}\label{Fig:3.2}
\end{figure}

Abe showed that for a classical knot diagram \( D \) \( \Delta_v ( D ) = 0 \) if and only if \( D \) is homogeneous \cite{Abe2011}. However, Abe's proof relies on \( T_O ( D ) \) containing no loops (an edge which begins and ends at the same vertex); classically, this is always the case as the oriented resolution is the alternately coloured resolution, so that \( T_O ( D ) \) is bipartite. Virtually, however, there are knots whose oriented resolution is not the alternately coloured resolution; this is explained fully below. An example is given in \Cref{Fig:3.2}. For now, it suffices to recall that the quantity \( \Delta_v \) can be expressed as the first Betti number of the graph, \( G_{O} \), defined as follows.

\begin{definition}\label{Def:gdelta}
	Let \( T_O ( D ) \) be associated to the virtual knot diagram \( D \). Form the graph \( G_{O} \) in the following way:
	\begin{enumerate}[(i)]
		\item For each connected component of \( T^+_O ( D ) \) place a vertex, and a vertex for each connected component of \( T^-_O ( D ) \).
		\item Each vertex of \( T_O ( D ) \) lies in exactly one connected component of \( T^+_O ( D ) \), and exactly one connected component of \( T^-_O ( D ) \). For each vertex of \( T_O ( D ) \) place an edge linking the vertices of \( G_\Delta \) corresponding to the connected components in which it lies.\CloseDef
	\end{enumerate}
\end{definition}

\begin{proposition}
	\label{Prop:removingselfedges}
	Let \( T_O ( D ) \) be associated to the virtual knot diagram \( D \), and \( \widetilde{T}_O ( D ) \) be a graph obtained from \( T_O ( D ) \) by adding or removing a loop (of arbitrary sign). Further, let \( \widetilde{G}_O \) be the graph formed from \( \widetilde{T}_O ( D ) \) following the method given in \Cref{Def:gdelta}, where \( \widetilde{T}^+_O ( D ) \) and \( \widetilde{T}^-_O ( D ) \) are formed in the obvious way. Then \( G_O = \widetilde{G}_O \).
\end{proposition}

\begin{proof}
	It is clear that
		\begin{equation*}
			\# \text{components} ( T^{+/-}_O ( D ) ) = \# \text{components} ( \widetilde{T}^{+/-}_O ( D ) )
		\end{equation*}
	(we have only added or removed a loop) so that
		\begin{equation*}
			\# \text{vertices} ( G_O ) = \# \text{vertices} ( \widetilde{G}_O ).
		\end{equation*}
	Further, as loops do not connect distinct vertices, two vertices are linked in \( G_O \) if and only if they are linked in \( \widetilde{G}_O \).
\end{proof}

In light of \Cref{Prop:removingselfedges} it is clear that we need only consider homogeneity of \( T_O ( D ) \) up to the addition and removal of loops.

\begin{definition}
	\label{Def:lhom}
	Let \( G \) be a signed graph and let \( \overline{G} \) be the graph formed by removing all loops of \( G \). Then \( G \) is \emph{l-homogenous} if \( \overline{G} \) is homogenous. A virtual knot diagram is l-homogenous if \( T_O ( D ) \) is, and a virtual knot is l-homogenous if it has an l-homogenous diagram.\CloseDef
\end{definition}

\begin{theorem}[Analogue of Theorem \( 3.4 \) of Abe \cite{Abe2011}]
	\label{Thm:deltazero}
	A virtual knot diagram \( D \) is l-homogeneous if and only if \( \Delta_v ( D ) = 0 \). Hence, for an l-homogeneous diagram \( D \) of a virtual knot \( K \)
	\begin{equation*}
		U ( D ) = s ( K ).
	\end{equation*}
\end{theorem}

\begin{proof}
	Abe's original proof yields the following statement: if \( D \) is such that \( T_O ( D ) \) is loopless, then \( D \) is homogenous if and only if \( \Delta_v ( D ) = 0 \). By \Cref{Prop:removingselfedges} we may remove any loops from \( T_O ( D ) \), leaving the associated \( G_{O} \) unchanged. Recalling that \( \Delta_v ( D ) = b_1 ( G_O ) \), the first Betti number of \( G_O \), we obtain the desired result.
\end{proof}

\subsection{The doubled Rasmussen invariant}
\label{Subsec:doubledbounds}
\subsubsection{Formulation}
\label{Subsec:doubledform}
In formulating the bounds on the doubled Rasmussen invariant we follow much the same path as in \Cref{Sec:dkkformulation}. The formulae arrived at in this section exhibit important differences between those of \Cref{Sec:dkkformulation}, however, owing to the structural differences between MDKK homology and doubled Lee homology.

We begin by making some preliminary definitions.

\begin{definition}
	Let \( D \) be a diagram of a virtual knot and \( G ( D ) \) its Gauss diagram. A classical crossing of \( D \), associated to the chord labelled \(c\) in \( G ( D ) \), is known as \emph{odd} if the number of chord endpoints appearing between the two endpoints of \( c \) is odd. Otherwise it is known as \emph{even}. The \emph{odd writhe} of \( D \) is defined
	\begin{equation*}
	J ( D ) = \sum_{\text{odd crossings of}~D} \text{sign of the crossing}.
	\end{equation*}\CloseDef
\end{definition}

\begin{theorem}
	Let \( D \) be a virtual knot diagram of \(K\). The odd writhe is an invariant of \( K \) and we define
	\begin{equation*}
	J ( K ) \coloneqq J ( D ).
	\end{equation*}
\end{theorem}

\begin{definition}
	\label{Def:acres}
	Let \( D \) be a virtual knot diagram. The \emph{alternately coloured resolution} of a classical crossing of \( D \) is the resolution it is smoothed into in the alternately colourable smoothing of \( D \).\CloseDef
\end{definition}

\begin{proposition}[Proposition \(4.11\) of \cite{Rushworth2017}]\label{Prop:unoriented}
	A classical crossing of a virtual knot diagram is odd if and only if it's alternately coloured resolution is the unoriented resolution.
\end{proposition}

In the construction of MDKK homology source-sink decorations are used to ensure that the oriented resolution of a virtual knot is, in fact, alternately colourable; doubled Khovanov homology does not do so. In the definition below, therefore, we consider the graph associated to the alternately coloured smoothing of a virtual knot.

\begin{definition}
	\label{Seifertgraph2}
	Given a virtual link diagram \( D \) denote by \( \mathscr{S} ( D ) \) the alternately coloured smoothing of \( D \). Denote by \( T_{\mathscr{S}} ( D ) \) the graph with a vertex for each cycle of \( \mathscr{S} ( D ) \) and an edge for each classical crossing of \( D \), decorated with the sign and parity of the crossing: every edge is decorated with an element of \( \lbrace (e,+),(e,-),(o,+),(o,-) \rbrace \), where \( (e,+) \) denotes an even positive crossing, \( (o,+) \) an odd positive crossing, and so on. The edge associated to a crossing is between the vertex or vertices associated to the cycles involved in the smoothing site of that crossing. The subgraph of \( T_{\mathscr{S}} ( D ) \) formed by removing all the edges labelled with either \( (e,+) \) or \( (o,-) \) is denoted \( T^{\Leftcircle}_{\mathscr{S}} ( D ) \). The subgraph of \( T_{\mathscr{S}} ( D ) \) formed by removing all the edges labelled with either \( (e,-) \) or \( (o,+) \) is denoted \( T^{\Rightcircle}_{\mathscr{S}} ( D ) \).\CloseDef
\end{definition}

\begin{definition}
	\label{Def:doubledu}
	Let \( D \) be a virtual knot diagram with \( n^o_+ \) (\( n^o_- \)) odd positive (negative) classical crossings. Define the quantities
	\begin{equation}
		\begin{aligned}
			U_d ( D ) &= \# \text{vertices} ( T_{\mathscr{S}} ( D ) ) - 2 \# \text{comp} ( T^{\Leftcircle}_{\mathscr{S}} ( D ) ) + wr ( D ) + J ( D ) + n^o_+ +1 \\
			\Delta_d ( D ) &= 2 ( \# \text{vertices} ( T_{\mathscr{S}} ( D ) ) - \# \text{comp} ( T^{\Leftcircle}_{\mathscr{S}} ( D ) ) - \# \text{comp} ( T^{\Rightcircle}_{\mathscr{S}} ( D ) ) + 1 ) \\ 
			& \quad + n^o_+ + n^o_-
		\end{aligned}
	\end{equation}
	where \( \# \text{comp} \) denotes the number of components of a graph.\CloseDef
\end{definition}

In direct analogy to \Cref{Thm:u} we have the following.

\begin{theorem}
	\label{Thm:doubledbounds}
	Let \( D \) be a diagram of a virtual knot \( K \). Then
	\begin{equation}
		U_d ( D ) - \Delta_d ( D ) \leq s_1 ( K ) \leq U_d ( D ).
	\end{equation} 
\end{theorem}

\begin{proof}
	We shall go through the proof of \Cref{Thm:doubledbounds} in more detail than that of it's counterpart \Cref{Thm:u}, owing to the aforementioned differences between the theories \( vKh ' \) and \( DKh ' \). The gist of the proof is unchanged, however: as computation of \( s_1 ( K ) \) only requires knowledge of the partial chain complex
		\begin{center}
			\begin{tikzpicture}[roundnode/.style={}]
		
				\node[roundnode] (s0)at (-5,0)  {\({\cdkh_{s_2(K)-1} ( D ) '}\)};
		
				\node[roundnode] (s1)at (0,0)  {\({\cdkh_{s_2(K)} ( D ) '}\)};
		
				\draw[->,thick] (s0)--(s1) node[above,pos=0.5]{\( d_{s_2(K)-1} \)};
				
			\end{tikzpicture}
		\end{center}
	we ignore (by resolving them) classical crossings whose alternately coloured resolution is the \(0\)-resolution; such crossings are associated to outgoing maps from the alternately coloured resolution of \( D \) and do not contribute to \( d_{s_2(K)-1} \). This comes at the price, of course: we lose a large amount of the information contained in \( {\cdkh '}_{s_2(K)} ( D ) \). Nevertheless, the trade is a worthwhile one, as we are able to use what's left to obtain bounds on \( s_1 ( K ) \).
	
	Let \( D \) be a diagram of a virtual knot \( K \), with \( n_+ \) (\(n_-\)) positive (negative) classical crossings. Further, let \( n_+ = n^e_+ + n^o_+ \) and \( n_- = n^e_- + n^o_- \), where a superscript \(e\) (\(o\)) denotes the number of even (odd) crossings. Form a virtual link diagram, \( \widetilde{D} \), by resolving all even positive crossings and all odd negative crossings of \( D \) into their alternately coloured resolutions. (One readily observes that such crossings are those with alternately coloured resolution the \(0\)-resolution, as mentioned above.) We can write
		\begin{equation*}
			\widetilde{D} = \widetilde{D}_1 \sqcup \widetilde{D}_2 \sqcup \ldots \sqcup \widetilde{D}_r
		\end{equation*}
	where \( \widetilde{D}_i \) is a virtual link diagram with \( n^i_+ \) positive and \( n^i_- \) negative classical crossings (the parity of positive (negative) crossings is necessarily odd (even), of course). Further, for \( \mathscr{S} \) the alternately colourable smoothing of \( D \), we have
		\begin{equation*}
			\mathscr{S} = \mathscr{S}_1 \sqcup \mathscr{S}_2 \sqcup \ldots \sqcup \mathscr{S}_r
		\end{equation*}
	where \( \mathscr{S}_i \) is the unique alternately colourable smoothing of \( \widetilde{D}_i \) formed by resolving all crossings into the resolution they are resolved into in \( \mathscr{S} \).
	
	Notice that while \( \cdkh ' ( D ) \) does not split as a tensor product of the \( \cdkh ' ( \widetilde{D}_i ) \)'s, the alternately coloured generators of \( \dkh ' ( K ) \) do. That is, if \( \sg^u \) is associated to \( \mathscr{S} \), then
		\begin{equation}
		\label{Eq:acgensplit}
			\sg^{\text{u}} = \sg^{\text{u}}_1 \otimes \sg^{\text{u}}_2 \otimes \cdots \otimes \sg^{\text{u}}_r
		\end{equation}
	where \( \sg^{\text{u}}_i \) is the alternately coloured generator defined by \( \mathscr{S}_i \).
	
	We have \( J ( \widetilde{D}_i ) = n^i_+ \) (as all negative crossings of \( \widetilde{D}_i \) are even), so that the highest non-trivial quantum grading of \( {\cdkh '}_{n^i_+} ( \widetilde{D}_i ) \) containing \( \left[ \sg^{\text{u}}_i \right] \) is \( e_i + n^i_+ + n^i_+ - n^i_- \), where \( e_i \) denotes the number of cycles of \( \mathscr{S}_i \). Further, as a corollary to Lemma \(4.2\) of \cite{Rushworth2017}, we determine that \( \left[ \sg^{\text{u}}_i \right] \) is not of top degree, and that \( e_i + n^i_+ + n^i_+ - n^i_- - 2 \) is the highest non-trivial degree of \( {\cdkh '}_{n^i_+} ( \widetilde{D}_i ) \) containing it. By \Cref{Eq:acgensplit} and an argument directly analogous to Lobb's \cite{Lobb2011} we obtain
		\begin{equation*}
			\begin{aligned}
				s^{\text{u}}_{\text{min}} ( K ) &\leq n^e_+ - n^o_- + \sum_{i=1}^{r} \left( e_i + n^i_+ + n^i_+ - n^i_- - 2 \right) \\
				&= wr ( D ) + J ( D ) + n^o_+ + \# \text{vertices} ( T_{\mathscr{S}} ( D ) ) - 2 \# \text{comp} ( T^{\Leftcircle}_{\mathscr{S}} ( D ) ).
			\end{aligned}
		\end{equation*}
	Recalling that \( s^{\text{u}}_{\text{min}} ( K ) = s_1 ( K ) + 1 \), we arrive at 
		\begin{equation*}
			s_1 ( K ) \leq U_d ( D ).
		\end{equation*}
	To see that
		\begin{equation*}
			U_d ( D ) - \Delta_d ( D ) \leq s_1 ( K ) 
		\end{equation*}
	repeat the proof of \Cref{Thm:delta}, which we are free to do as the doubled Rasmussen invariant replicates the behaviour of its classical counterpart with respect to the	 mirror image.
\end{proof}

\subsubsection{Simplifying \( \Delta_d ( D ) \)}\label{sec:ddeltazero}

Much of the analysis used in the \Cref{Sec:deltazero} may be repeated in order to characterise a case in which the \( \Delta_d \) formula simplifies. However, we do not recover the vanishing result as in the case of \( \Delta_v \).

\begin{definition}
	\label{Def:dhom}
	Let \( D \) be a virtual knot diagram and \( T_{\mathscr{S}} ( D ) \) the graph associated to it. Recall that each edge of \( T_{\mathscr{S}} ( D ) \) is decorated with exactly one element of
	\begin{equation*}
	\lbrace (e,+),(e,-),(o,+),(o,-) \rbrace.
	\end{equation*}
	Let \( \Rightcircle = \lbrace (e,-),(o,+) \rbrace \) and \( \Leftcircle = \lbrace (e,+),(o,-) \rbrace \). The graph \( T_{\mathscr{S}} ( D ) \) is \emph{d-homgenous} if every block is decorated with elements of either \( \Rightcircle \) or \( \Leftcircle \), but not both.

	The diagram \( D \) is d-homogenous if \( T_{\mathscr{S}} ( D ) \) is d-homogenous. A virtual knot is d-homogenous if it has a d-homogenous diagram.\CloseDef
\end{definition}

\begin{proposition}
	\label{Prop:ddelta}
	Let \( D \) be a virtual link diagram and \( T_{\mathscr{S}} ( D ) \) the graph associated to it. Then \( D \) is d-homogenous if and only if
	\begin{equation*}
		\# \text{vertices} ( T_{\mathscr{S}} ( D ) ) - \# \text{comp} ( T^{\Leftcircle}_{\mathscr{S}} ( D ) ) - \# \text{comp} ( T^{\Rightcircle}_{\mathscr{S}} ( D ) ) + 1 = 0.
	\end{equation*}	
\end{proposition}

\begin{proof}
	Let \( G_{\mathscr{S}} \) denote the graph formed from \( T_{\mathscr{S}} ( D ) \) in direct analogy to \( G_O \), as given in \Cref{Def:gdelta}, with \( T^{\Leftcircle}_{\mathscr{S}} ( D ) \) and \( T^{\Rightcircle}_{\mathscr{S}} ( D ) \) taking the place of \( T^+_O ( D ) \) and \( T^-_O ( D ) \). The graph \( T_{\mathscr{S}} ( D ) \) is bipartite as \( \mathscr{S} ( D ) \) is alternately coloured. Thus it is loopless and Abe's proof may be employed to show that \( T_{\mathscr{S}} ( D ) \) is homogenous if and only if \( b_1 ( G_{\mathscr{S}} ) = 0 \). We conclude by noticing that
	\begin{equation*}
		b_1 ( G_{\mathscr{S}} ) = \# \text{vertices} ( T_{\mathscr{S}} ( D ) ) - \# \text{comp} ( T^{\Leftcircle}_{\mathscr{S}} ( D ) ) - \# \text{comp} ( T^{\Rightcircle}_{\mathscr{S}} ( D ) ) + 1,
	\end{equation*}
	which follows exactly as in the case of \( \Delta_v \) and \( G_O \).
\end{proof}

\begin{corollary}
	\label{Cor:ddeltamin}
	Let \( D \) be diagram of a virtual knot \( K \). If \(D\) is d-homogenous then
	\begin{equation*}
		U_d ( D ) - n^o_+ - n^o_- \leq s_1 ( K ) \leq U_d ( D )
	\end{equation*}
	where \( n^o_+ \) (\( n^o_- \)) denotes the number of odd positive (negative) classical crossings of \( D \).
\end{corollary}

\section{Computation and estimation of the slice genus}\label{Sec:applications}

In this section we use the bounds \( U_v \) and \( U_d \) to compute or estimate the slice genus of a number of virtual knots. The computations are made by finding a surface of appropriate genus between the given knot and the unknot.

The following table contains the results of the analysis for the virtual knots of \(4\) crossing or less in Green's table \cite{Green}. A blank entry denotes an unknown, and most computations of \( s \), \( s_1 \), and \( s_2 \) (or the interval in which they lie) are made by computing \( U_{v/d} \), \( \Delta_{v/d} \), and \( J \) for the diagram given in the table. The exceptions to this are \( s_1 ( 3.3 ) \), which the author computed by hand from \( \dkh ' ( 3.3 ) \), and leftmost knots, for which the definition and the method of computation of \( s_1 \) are given in \cite[Section \(4.4\)]{Rushworth2017}. Further, many computations of \( s \), \( s_2 \), and \( s_2 \) are made by spotting that the knot in question is a connect sum of two other knots, and employing the additivity of the invariants along with their invariance under flanking \cite[Definition \(2.6\)]{Rushworth2017}. (As observed in \Cref{Subsec:evenknots}, \(s\) and \(s_1\) coincide for even knots, so that the invariants are buy one get one free in this case.)

Exact values of \( g^\ast \) are obtained by constructing a cobordism which attains a lower bound given by \( s \), \( s_1 \), or \( s_2 \). Upper bounds on \( g^\ast \) are obtained by constructing a cobordism of the given genus, and employing the fact that half the crossing number bounds the slice genus of a knot from above (as in the classical case) \cite{Boden2017}. Shortly after posting a previous version of this paper to the arXiv the author learned of the work of Boden, Chrisman, and Gaudreau in which they compute or estimate the slice genus of a very large number of the virtual knots of \(6\) crossings or less \cite{Boden2017,BCGtable}. In the table below we do not include the values of \( g^\ast \) they arrive at in order to demonstrate the infomation that can be obtained using the bounds \( U_v \), \( U_d \), and the properties of the virtual and doubled Rasmussen invariants.

\begin{center}\label{Tab:slicegenera}
	\begin{filecontents*}{vknotinfo_v2.csv}
Virtual knot,positive,negative,leftmost,l-hom.,d-hom.,$s$,$s_1$,$s_2$,$g$
0.1,Y,Y,Y,Y,Y,0,0,0,0
2.1,,Y,Y,Y,Y,-2,-5,-2,1
3.1,,,,,,{$[-2,0]$},{$[-3,1]$},0,{$[0,2]$}
3.2,,,Y,Y,,0,-4,-2,1
3.3,,Y,,Y,,-2,-6,-2,1
3.4,,,Y,,,{$[-2,0]$},-4,-2,1
3.5,,Y,,,,-2,-2,0,1
3.6,,Y,,Y,Y,-2,-2,0,1
3.7,,,,,,0,0,0,{$[0,2]$}
4.1,,Y,Y,Y,,-4,-10,-4,2
4.2,,,,Y,,0,0,0,{$[0,2]$}
4.3,,Y,Y,Y,,-4,-10,-4,2
4.4,,,,Y,,-2,-5,-2,1
4.5,,,,Y,,-2,-5,-2,1
4.6,,,,Y,Y,0,0,0,{$[0,2]$}
4.7,,Y,Y,Y,,-4,-10,-4,2
4.8,,,,Y,Y,0,0,0,0
4.9,,Y,,Y,,-4,{$[-9,-5]$},-2,2
4.10,,,,Y,,-2,{$[-4,0]$},0,{$[1,2]$}
4.11,,,,Y,,-2,{$[-7,-2]$},-2,{$[1,2]$}
4.12,,,,Y,,0,{$[-2,2]$},0,{$[0,2]$}
4.13,,,,Y,,0,{$[-2,2]$},0,{$[0,2]$}
4.14,,,Y,Y,,0,-3,-2,{$[1,2]$}
4.15,,Y,,Y,,-4,{$[-9,-5]$},-2,2
4.16,,,,Y,,-2,{$[-4,0]$},0,1
4.17,,,,Y,,0,{$[-3,1]$},0,{$[0,2]$}
4.18,,,Y,Y,,-2,-5,-2,1
4.19,,,,,,,,,{$[0,2]$}
4.20,,,Y,Y,Y,0,-3,-2,1
4.21,,,,Y,,0,{$[0,2]$},0,{$[0,2]$}
4.22,,,,Y,Y,0,{$[-5,-2]$},-2,{$[1,2]$}
4.23,,,,Y,,-2,{$[-4,0]$},0,{$[1,2]$}
4.24,,,,Y,,2,{$[0,4]$},0,{$[1,2]$}
4.25,,Y,Y,Y,Y,-2,-9,-4,1
4.26,,,,,,{$[-2,0]$},{$[-5,2]$},0,{$[0,2]$}
4.27,,,,Y,,0,{$[-6,-2]$},-2,1
4.28,,,,,,{$[-2,0]$},{$[-5,2]$},0,{$[0,2]$}
4.29,,Y,,Y,,-4,{$[-11,-5]$},-2,2
4.30,,,,Y,,-2,{$[-9,-2]$},-2,1
4.31,,,,Y,,-2,{$[-6,0]$},0,1
4.32,,,,Y,,-2,{$[-6,0]$},-2,{$[1,2]$}
4.33,,,,Y,,-2,{$[-9,-2]$},-2,1
4.34,,,Y,Y,Y,0,-3,-2,{$[1,2]$}
4.35,,,,Y,,0,{$[-4,2]$},0,{$[0,1]$}
4.36,,,,Y,,0,{$[1,7]$},2,{$[1,2]$}
4.37,,Y,,Y,,-4,{$[-11,-5]$},-2,2
4.38,,,,Y,,-2,{$[-9,-2]$},-2,1
4.39,,,,Y,,-2,{$[-8,-2]$},-2,{$[1,2]$}
4.40,,,Y,Y,Y,0,-3,-2,1
4.41,,,,Y,,-2,{$[-6,0]$},0,1
4.42,,,,Y,,0,{$[-2,2]$},0,{$[0,2]$}
4.43,,Y,Y,Y,Y,-2,-9,-4,1
4.44,,,,,,{$[-2,0]$},{$[-10,-2]$},-2,1
4.45,,,,,,{$[-2,0]$},{$[-10,-2]$},-2,1
4.46,,,,,,{$[-2,0]$},{$[-4,4]$},0,{$[0,2]$}
4.47,,,,,,{$[0,2]$},{$[-5,2]$},0,{$[0,2]$}
4.48,,Y,,Y,,-4,{$[-13,-5]$},-2,2
4.49,,,,Y,,-2,{$[-9,-2]$},-2,{$[0,2]$}
4.50,,,,Y,,-2,{$[-6,0]$},0,{$[1,2]$}
4.51,,,,Y,,-2,{$[-6,0]$},0,{$[1,2]$}
4.52,,,Y,Y,Y,0,{$[-5,-2]$},-2,1
4.53,,Y,Y,Y,Y,-4,-10,-4,2
4.54,,,,Y,,-2,-10,-2,{$[1,2]$}
4.55,,,,Y,,0,0,0,{$[0,2]$}
4.56,,,,Y,,0,0,0,{$[0,1]$}
4.57,,,,Y,,-2,{$[-6,0]$},0,{$[1,2]$}
4.58,,,,Y,,0,{$[-4,2]$},0,{$[0,1]$}
4.59,,,,Y,,0,{$[-4,2]$},0,{$[0,1]$}
4.60,,,Y,Y,Y,0,-3,-2,1
4.61,,Y,,Y,,-4,{$[-10,-6]$},-2,2
4.62,,,,Y,,-2,{$[-7,-2]$},-2,{$[1,2]$}
4.63,,,,Y,,-4,{$[-7,-2]$},-2,2
4.64,,,Y,Y,Y,-4,-3,-2,2
4.65,,,,Y,,-2,{$[-4,0]$},0,{$[1,2]$}
4.66,,,,Y,,0,{$[-2,2]$},0,{$[1,2]$}
4.67,,,,Y,,0,{$[-2,2]$},0,{$[0,1]$}
4.68,,,,Y,,2,{$[-4,0]$},0,1
4.69,,,,Y,,-4,{$[-9,-5]$},-2,2
4.70,,,,Y,,-2,{$[-4,0]$},0,1
4.71,,,,Y,,0,0,0,0
4.72,,,,Y,,0,0,0,0
4.73,,,,,,-4,-10,-4,2
4.74,,,,,,-2,-5,-2,1
4.75,,,,,,0,0,0,0
4.76,,,,,,0,0,0,0
4.77,,,,,,0,0,0,0
4.78,,,,Y,,-4,{$[-9,-5]$},-2,2
4.79,,,,Y,,-2,{$[-4,0]$},0,{$[1,2]$}
4.80,,,Y,Y,Y,-2,-9,-4,1
4.81,,,,Y,,0,{$[-8,-2]$},-2,{$[1,2]$}
4.82,,,Y,Y,Y,-2,-8,-2,1
4.83,,,,,,{$[-4,0]$},{$[-8,-4]$},-2,{$[1,2]$}
4.84,,,,,Y,{$[0,2]$},{$[-2,0]$},2,{$[1,2]$}
4.85,,,,,,{$[-2,2]$},{$[-2,2]$},0,{$[0,1]$}
4.86,,,,Y,Y,0,0,0,{$[0,1]$}
4.87,,,,Y,Y,-2,{$[-8,-6]$},-2,{$[1,2]$}
4.88,,,,,Y,{$[0,2]$},{$[4,6]$},2,1
4.89,,,,Y,Y,-2,-2,0,1
4.90,,,,Y,Y,0,0,0,0
4.91,,,Y,Y,Y,-4,-11,-4,2
4.92,,,,Y,Y,-2,{$[-8,-6]$},-2,1
4.93,,,,,,{$[-2,0]$},{$[-3,1]$},0,{$[0,1]$}
4.94,,,,Y,Y,-2,{$[-8,-6]$},-2,1
4.95,,,,,,{$[-2,0]$},{$[-8,-4]$},-2,1
4.96,,,,,,{$[-2,0]$},{$[-3,1]$},0,{$[0,1]$}
4.97,,,,,,{$[-2,2]$},{$[-3,2]$},0,{$[0,1]$}
4.98,,,,,,{$[-2,2]$},{$[-2,2]$},0,{$[0,1]$}
4.99,,,,Y,Y,0,0,0,{$[0,1]$}
4.100,,,,Y,,0,{$[-3,2]$},0,{$[0,1]$}
4.101,,,,Y,,-2,{$[-10,-4]$},-2,{$[1,2]$}
4.102,,,,Y,,0,{$[-3,2]$},0,{$[0,2]$}
4.103,,,,,,{$[-2,0]$},{$[-3,1]$},0,{$[0,2]$}
4.104,,,,,Y,{$[0,2]$},{$[4,6]$},2,1
4.105,,,,,,-2,-2,0,1
4.106,,,,,,{$[-2,2]$},{$[-2,2]$},0,{$[0,1]$}
4.107,,,,,,{$[-2,2]$},{$[-2,2]$},0,{$[0,1]$}
4.108,,,,,,0,0,0,1
\end{filecontents*}
\csvreader[head to column names,longtable=|c|c|c|c|c|c|c|c|,table head=\hline Knot & l-hom. & d-hom. & \(s\) & \(s_1\) & \(s_2\) & \(g^\ast\) \\\hline\endhead\hline\endfoot,table foot=\hline]{vknotinfo_v2.csv}%
	{1=\vknot,5=\lhom,6=\dhom,7=\sinv,8=\soneinv,9=\stwoinv,10=\sgenus}%
	{\vknot & \lhom & \dhom & \sinv & \soneinv & \stwoinv & \sgenus}
\end{center}

From the table we are able to make some observations regarding the two extensions of the Rasmussen invariant. We see that only \( s_1 \) is able to distinguish between \( 2.1 \) and \( 3.3 \). Further, there are a number of knots for which the easy to compute \( s_2 \) obstructs sliceness while the harder to compute \( s \) does not. The virtual and doubled Rasmussen invariants are also able to distinguish many pairs of knots which have the same positive slice genus, showing that they are not concordant to one another.

We also give presentations of the surfaces of genus \(0\), \(1\), and \(2\) used to determine the slice genus of the knots \(4.8\), \(3.5\), and \( 4.15 \) respectively; they are contained in \Cref{Fig:48slicedisc,Fig:35surface,Fig:415surface}. Unlabeled arrows denote virtual Reidemeister moves, while those which denote \(1\)-handle additions are so labelled. Red arcs between strands denote the locations of such handle additions within individual diagrams.

\begin{figure}
	\includegraphics[scale=0.75]{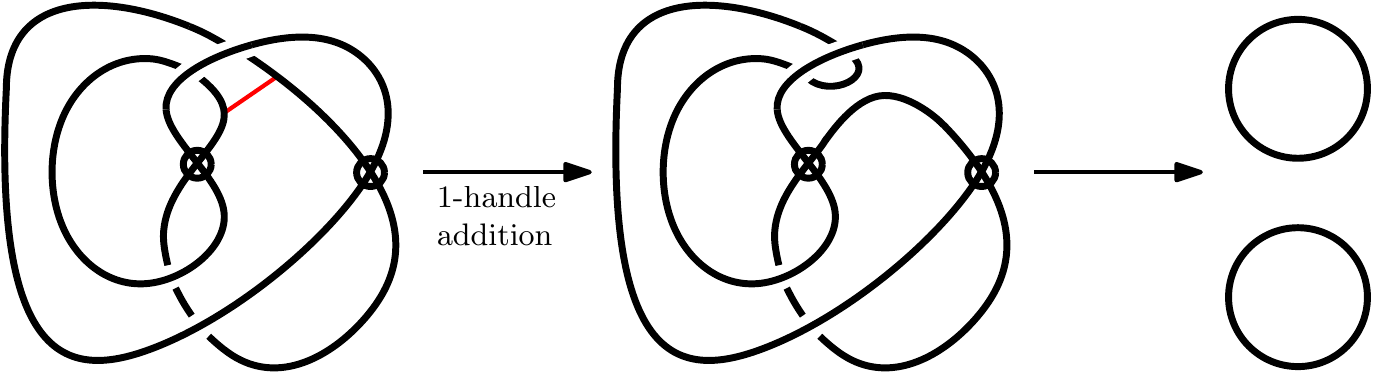}
	\caption{A slice disc for virtual knot \(4.8\).}
	\label{Fig:48slicedisc}
\end{figure}

\begin{figure}
	\includegraphics[scale=0.75]{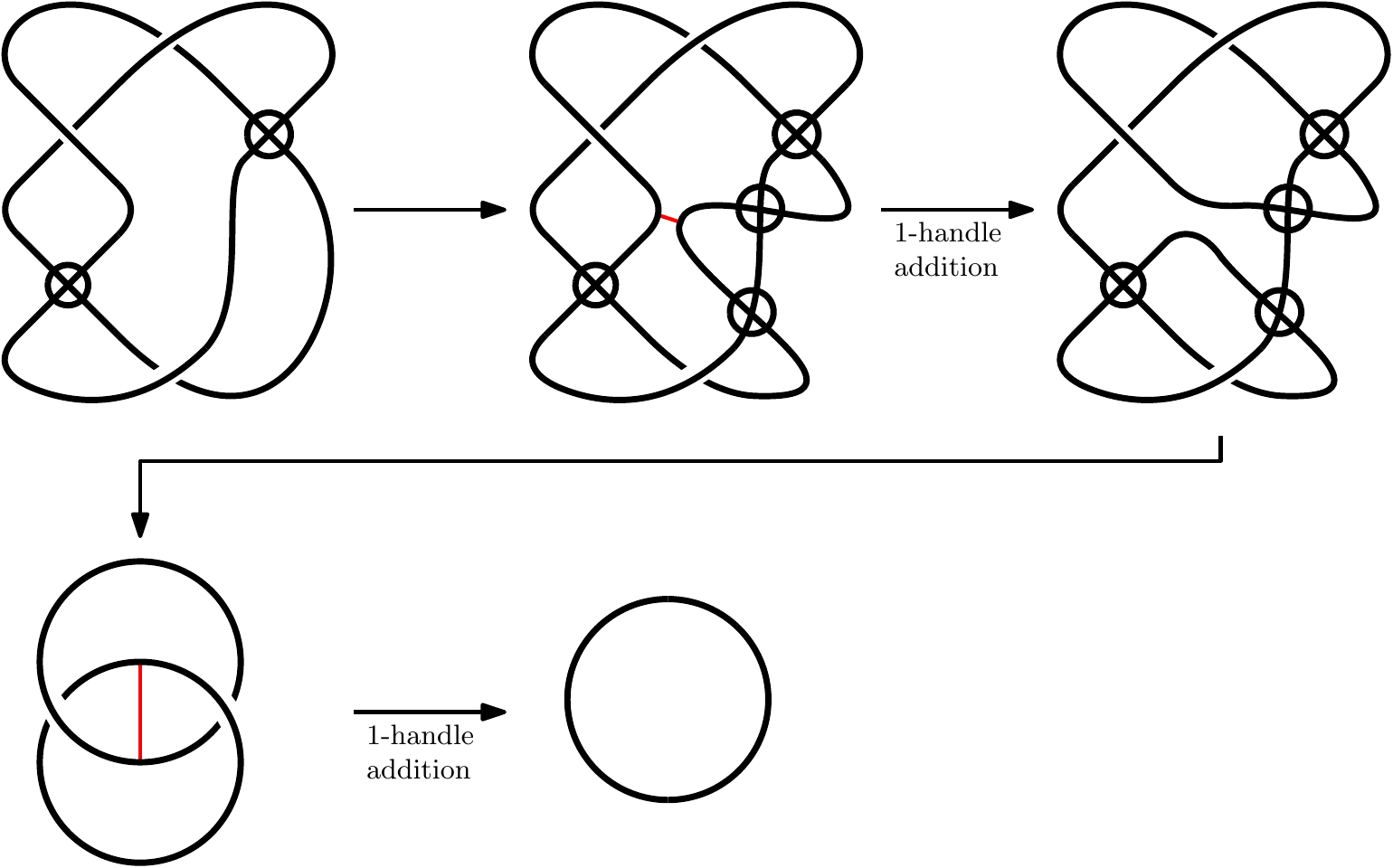}
	\caption{A genus \(1\) cobordism to the unknot from virtual knot \(3.5\).}
	\label{Fig:35surface}
\end{figure}

\begin{figure}
	\includegraphics[scale=0.75]{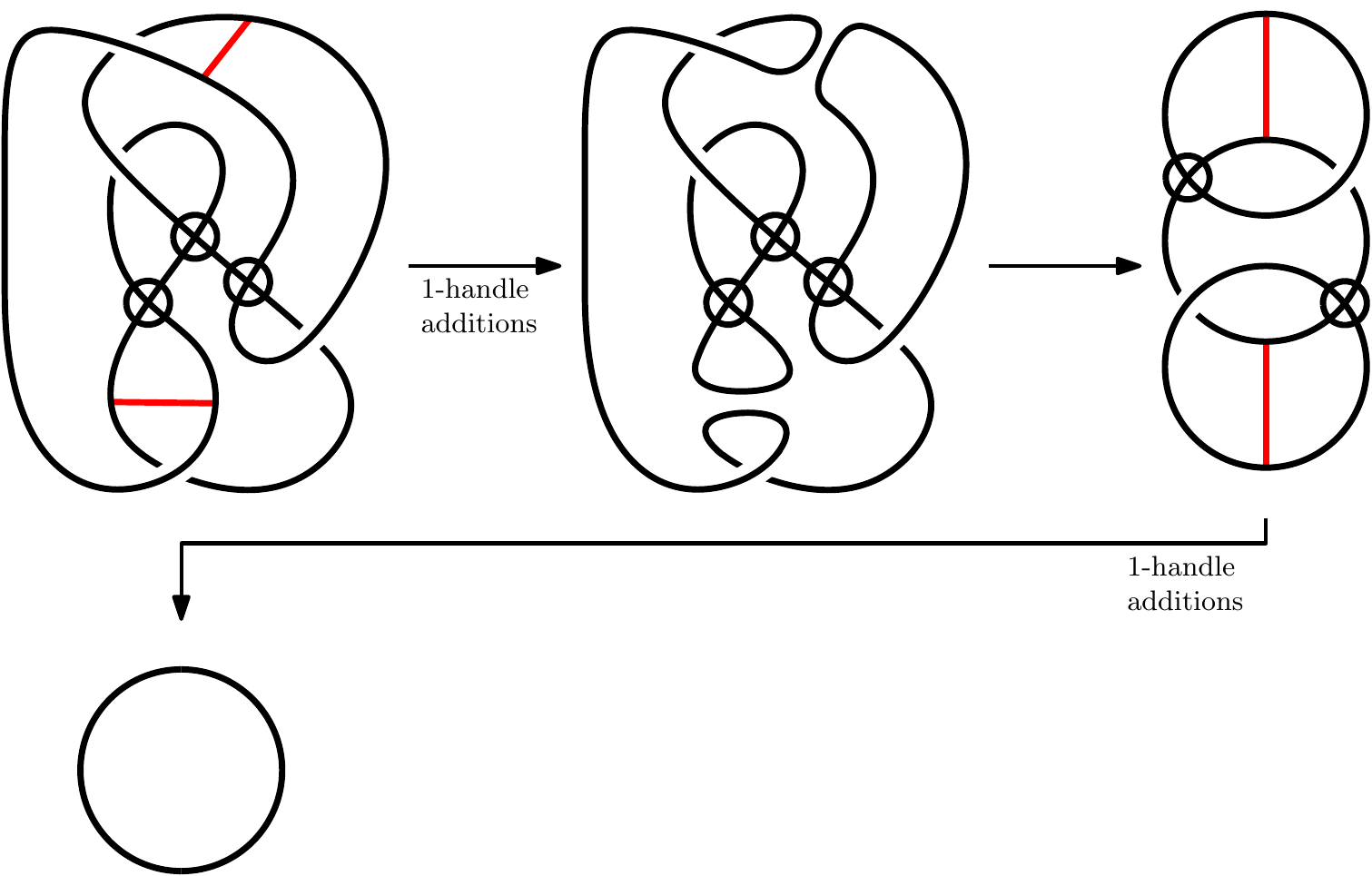}
	\caption{A genus \(2\) cobordism to the unknot from virtual knot \(4.15\).}
	\label{Fig:415surface}
\end{figure}

To conclude we list the results of similar analysis as that used to produce the previous table, this time on the virtual knots for which Boden, Chrisman, and Gaudreau's methods are unable to obstruct sliceness but the virtual and doubled Rasmussen invariants can. The upper bounds on \( g^\ast \) are those given by Boden, Chrisman, and Gaudreau \cite{BCGtable}. As in the case of knots of \( 4 \) or less crossings many of the computations are made by spotting connect sums.

\begin{center}\label{Tab:slicegaps}
	\begin{filecontents*}{vknotinfo_gaps.csv}
Virtual knot,positive,negative,leftmost,l-hom.,d-hom.,$s$,$s_1$,$s_2$,$g$
5.114,,,,,,-2,-1,0,{$[1,2]$}
5.344,,,,,,2,9,0,{$[1,2]$}
5.2351,,,,Y,Y,-2,{$[-3,1]$},0,{$[1,2]$}
6.1617,,,,,,-2,-1,0,{$[1,2]$}
6.2414,,,,,,-2,{$[-2,3]$},0,{$[1,2]$}
6.3036,,,,,,0,{$[1,3]$},0,1
6.3452,,,,,,2,{$[0,4]$},0,1
6.3536,,,,,,-2,{$[-6,0]$},0,1
6.3537,,,,,,-2,{$[-6,0]$},0,1
6.3780,,,,,,-2,{$[-6,0]$},0,1
6.3781,,,,,,-2,{$[-6,0]$},0,1
6.3972,,,,,,-2,{$[-6,0]$},0,1
6.3973,,,,,,-2,{$[-6,0]$},0,1
6.5252,,,,,,-2,{$[-6,0]$},0,1
6.5253,,,,,,-2,{$[-6,0]$},0,1
6.5738,,,,,,-2,{$[-6,0]$},0,1
6.5740,,,,,,-2,{$[-6,0]$},0,1
6.6176,,,,Y,,-2,{$[-4,0]$},0,{$[1,2]$}
6.6508,,,,,,-2,{$[-6,0]$},0,1
6.6509,,,,,,-2,{$[-6,0]$},0,1
6.7805,,,,,,-2,{$[-6,0]$},0,1
6.7807,,,,,,-2,{$[-6,0]$},0,1
6.8909,,,,,,0,-1,0,{$[1,2]$}
6.9825,,,,,,0,-1,0,{$[1,2]$}
6.12069,,,,,,2,{$[-1,3]$},0,{$[1,2]$}
6.13061,,,,,,2,{$[-1,3]$},0,{$[1,2]$}
6.14012,,,,Y,,2,{$[-3,3]$},0,1
6.28566,,,,,,0,-1,0,{$[1,2]$}
6.35229,,,,Y,Y,2,{$[0,4]$},0,{$[1,2]$}
6.37329,,,,,,0,3,0,{$[1,2]$}
6.37570,,,,Y,,2,{$[-1,5]$},0,1
6.38605,,,,Y,,2,{$[-2,4]$},0,1
6.42015,,,,Y,,-2,{$[-4,0]$},0,1
6.46580,,,,Y,,2,{$[-4,4]$},0,1
6.46684,,,,Y,,2,{$[-4,4]$},0,1
6.49730,,,,Y,,-2,{$[-4,0]$},0,1
6.58375,,,,,,0,-3,0,1
6.58930,,,,Y,,2,{$[0,4]$},0,1
6.70672,,,,Y,Y,-2,-2,0,{$[1,2]$}
6.75192,,,,Y,,2,{$[0,4]$},0,1
6.78145,,,,Y,,-2,{$[-4,0]$},0,1
6.85784,,,,Y,,-2,-2,0,{$[1,2]$}
6.90115,,,,Y,Y,-2,-2,0,{$[1,2]$}
6.90150,,,,Y,Y,-2,-2,0,{$[1,2]$}
6.90209,,,,Y,Y,-2,-2,0,{$[1,2]$}
\end{filecontents*}
\csvreader[head to column names,longtable=|c|c|c|c|c|c|c|,table head=\hline Knot & l-hom. & d-hom. & \(s\) & \(s_1\) & \(s_2\) & \(g^\ast\) \\\hline\endhead\hline\endfoot,table foot=\hline]{vknotinfo_gaps.csv}%
	{1=\vknot,5=\lhom,6=\dhom,7=\sinv,8=\soneinv,9=\stwoinv,10=\sgenus}%
	{\vknot & \lhom & \dhom & \sinv & \soneinv & \stwoinv & \sgenus}
\end{center}

\bibliographystyle{plain}
\bibliography{library}

\end{document}